\documentclass{siamart1116}

\usepackage{lipsum}
\usepackage{amsfonts}
\usepackage{graphicx}
\usepackage{epstopdf}
\usepackage{algorithmic}

\usepackage{amsopn}
\usepackage{amssymb}
\usepackage{amsmath}
\usepackage{graphicx}
\usepackage{caption}
\usepackage{subcaption}
\usepackage{caption}
\usepackage{float}
\usepackage{subcaption}
\usepackage{dcolumn}
\usepackage{multirow}
\usepackage{enumitem}
\usepackage{amsfonts}

\usepackage[sort]{cite}

\usepackage{color}

\ifpdf
\DeclareGraphicsExtensions{.eps,.pdf,.png,.jpg}
\else
\DeclareGraphicsExtensions{.eps}
\fi

\numberwithin{theorem}{section}

\newcommand{\TheTitle}{DOMAIN DECOMPOSITION AND MULTISCALE MORTAR MIXED FINITE ELEMENT
METHODS FOR LINEAR ELASTICITY WITH WEAK STRESS SYMMETRY}
\newcommand{\TheHeader}{Domain decomposition and mortar mixed methods for elasticity}
\newcommand{\TheAuthors}{Eldar Khattatov and Ivan Yotov}

\headers{\TheHeader}{\TheAuthors}

\title{{\TheTitle}
\thanks{\today.
\funding{NSF grant DMS 1418947 and DOE grant DE-FG02-04ER25618.}
}
}

\author{
Eldar Khattatov\thanks{Department of Mathematics, University
of Pittsburgh, Pittsburgh, PA 15260, USA
(\email{elk58@pitt.edu}, \email{yotov@math.pitt.edu}).}
\and
Ivan Yotov\footnotemark[2]
}

\allowdisplaybreaks

\newcommand{\inp}[2][]{\left(#1, #2\right)}
\newcommand{\gnp}[2][]{\left\langle#1, #2\right\rangle}
\newcommand{\dvr}[1]{\nabla \cdot #1}

\def\skew{\operatorname{skew}}

\newtheorem{remark}{Remark}[section]

\def\meas{\operatorname{meas\,}}
\def\tr{\operatorname{tr\,}}

\def\skew{\operatorname{Skew}}
\def\curl{\operatorname{curl}}
\def\dvr{\operatorname{div}}

\def\s{\sigma}
\def\t{\tau}

\def\g{\gamma}
\def\l{\lambda}
\def\m{\mu}
\def\e{\epsilon}

\newcommand{\sh}[1]{\s_{h,#1}}
\newcommand{\uh}[1]{u_{h,#1}}
\newcommand{\gh}[1]{\g_{h,#1}}

\def\ub{\bar{u}}
\def\gb{\bar{\g}}

\def\ss{\s^{*}}
\def\us{u^{*}}
\def\gs{\g^{*}}

\def\Pit{\tilde{\Pi}}

\def\X{\mathbb{X}}
\def\W{\mathbb{W}}

\def\R{\mathbb{R}}
\def\RM{\mathbb{RM}}
\def\M{\mathbb{M}}
\def\N{\mathbb{N}}
\def\S{\mathbb{S}}

\def\P{\mathcal{P}}
\def\Q{\mathcal{Q}}
\def\Rc{\mathcal{R}}

\def\BDM{\mathcal{BDM}}
\def\T{\mathcal{T}}
\def\IHc{\mathcal{I}_H^c}

\def\O{\Omega}
\def\dO{\partial\Omega}
\def\G{\Gamma}
\def\Rd{\mathbb{R}^d}
\def\L2{L^2}
\def\H1{H^1}

\def\Th{\mathcal{T}}

\numberwithin{equation}{section}

\ifpdf
\hypersetup{
pdftitle={\TheTitle},
pdfauthor={\TheAuthors}
}
\fi

\begin{document}

\maketitle

\begin{abstract}
Two non-overlapping domain decomposition methods are presented for the
mixed finite element formulation of linear elasticity with weakly
enforced stress symmetry. The methods utilize either displacement or
normal stress Lagrange multiplier to impose interface continuity of
normal stress or displacement, respectively. By eliminating the
interior subdomain variables, the global problem is reduced to an
interface problem, which is then solved by an iterative procedure.
The condition number of the resulting algebraic interface problem is
analyzed for both methods. A multiscale mortar mixed finite element
method for the problem of interest on non-matching multiblock grids is
also studied. It uses a coarse scale mortar finite element space on
the non-matching interfaces to approximate the trace of the
displacement and impose weakly the continuity of normal stress. A
priori error analysis is performed. It is shown that, with appropriate
choice of the mortar space, optimal convergence on the fine scale is
obtained for the stress, displacement, and rotation, as well as some
superconvergence for the displacement. Computational results are
presented in confirmation of the theory of all proposed methods.
\end{abstract}



\section{Introduction}

Mixed finite element (MFE) methods for elasticity are important
computational tools due to their local momentum conservation, robust
approximation of the stress, and non-locking behavior for almost
incompressible materials. In this paper, we focus on MFE methods with
weakly imposed stress symmetry
\cite{Amara-Thomas,PEERS,stenberg1988family,arnold2007mixed,boffi2009reduced,cockburn2010new,gopalakrishnan2012second,ArnAwaQiu,Awanou-rect-weak},
since they allow for spaces with fewer degrees of freedom, as well as
reduction to efficient finite volume schemes for the displacement
\cite{MSMFEM-1,MSMFEM-2}. We note that the developments in this paper also
apply to MFE methods for elasticity with strong stress symmetry.

In many physical applications, obtaining the desired resolution may
result in a very large algebraic system. Therefore a critical
component for the applicability of MFE methods for elasticity is the
development of efficient techniques for the solution of these
algebraic systems. Domain decomposition methods
\cite{Toselli-Widlund,QV} provide one such approach. They adopt the
"divide and conquer" strategy and split the computational domain into
multiple non-overlapping subdomains. Then, solving the local problems
of lower complexity with an appropriate choice of interface conditions
leads to recovering the global solution. This approach naturally leads
to designing parallel algorithms, and also allows for the reuse of
existing codes for solving the local subdomain problems.
Non-overlapping domain decomposition methods for non-mixed
displacement-based elasticity formulations have been studied
extensively
\cite{Klawonn-dd-elast,Kim-elast-BDDC,Kim-elast-FETI,Fritz-mortar-elast,GPWW-2009,Hauret-LeTallec},
see also \cite{Goldfeld-dd-elast-mixed,Pavarino-dd-elast-mixed} for
displacement-pressure mixed formulations. To the best of our
knowledge, non-overlapping domain decomposition methods for
stress-displacement mixed elasticity formulations have not been
studied.

In this paper, we develop two non-overlapping domain decomposition
methods for the mixed finite element discretization of linear
elasticity with weakly enforced stress symmetry. The first method uses
a displacement Lagrange multiplier to impose interface continuity of
the normal stress. The second method uses a normal stress Lagrange
multiplier to impose interface continuity of the displacement. These
methods can be thought of as elasticity analogs of the methods
introduced in \cite{GW} for scalar second order elliptic problems, see
also \cite{cowsar1995balancing}. In both methods, the global system is reduced to an
interface problem by eliminating the interior subdomain variables.  We
show that the interface operator is symmetric and positive definite,
so the interface problem can be solved by the conjugate gradient
method. Each iteration requires solving Dirichlet or Neumann subdomain
problems.  The condition number of the resulting algebraic interface
problem is analyzed for both methods, showing that it is
$O(h^{-1})$. We note that in the second method the Neumann subdomain
problems can be singular. We deal with floating subdomains by
following the approach from the FETI methods
\cite{farhat1991method,Toselli-Widlund}, solving a coarse space problem to ensure
that the subdomain problems are solvable.

We also develop a multiscale mortar mixed finite element method for
the domain decomposition formulation of linear elasticity with
non-matching grids. We note that domains with complex geometries can be
represented by unions of subdomains with simpler shapes that are
meshed independently, resulting in non-matching grids across the
interfaces. The continuity conditions are imposed using mortar finite
elements, see
e.g. \cite{arbogast2000mixed,Fritz-mortar-elast,Kim-elast-FETI,Kim-elast-BDDC,GPWW-2009,pencheva2003balancing,Hauret-LeTallec}. Here
we focus on the first formulation, using a mortar finite element space
on the non-matching interfaces to approximate the trace of the
displacement and impose weakly the continuity of normal stress. We
allow for the mortar space to be on a coarse scale $H$, resulting in a
multiscale approximation, see
e.g. \cite{PWY,APWY,ganis2009implementation}. A priori error analysis
is performed. It is shown that, with appropriate choice of the mortar
space, optimal convergence on the fine scale is obtained for the
stress, displacement, and rotation, as well as some superconvergence
for the displacement.

The rest of the paper is organized as follows. The problem of
interest, its MFE approximation, and the two domain decomposition
methods are formulated in Section 2. The analysis of the resulting
interface problems is presented in Section 3. The multiscale mortar
MFE element method is developed and analyzed in Section 4.  A
multiscale stress basis implementation for the interface problem is
also given in this section. The paper concludes with computational
results in Section 5, which confirm the theoretical results on the
condition number of the domain decomposition methods and the
convergence of the solution of the multiscale mortar MFE element
method.

\section{Formulation of the methods}
\subsection{Model problem}
Let $\O \subset \mathbb{R}^d$, $d=2,3$ be a simply connected bounded
polygonal domain occupied by a linearly elastic body. Let $\M$, $\S$, and $\N$
be the spaces of $d\times d$ matrices, symmetric matrices, and
skew-symmetric matrices over the field $\mathbb{R}$, respectively. The
material properties are described at each point $x \in \O $ by a
compliance tensor $A = A(x)$, which is a self-adjoint, bounded, and
uniformly positive definite linear operator acting from $\S$ to
$\S$. We assume that $A$ can be extended to an operator from $\M$ to
$\M$ with the same properties. In particular, in the case of
homogeneous and isotropic body,
\begin{equation}\label{A-defn}
A\sigma = \frac{1}{2\mu} \left( \sigma - \frac{\lambda}{2\mu +
d\lambda}\operatorname{tr}(\sigma)I \right),
\end{equation}
where $I$ is the $d \times d$ identity matrix and $\mu > 0, \lambda
\ge 0$ are the Lam\'{e} coefficients.

Throughout the paper the divergence operator is the usual divergence
for vector fields, which produces a vector field when applied to a matrix
field by taking the divergence of each row. We will also use the curl
operator which is the usual curl when applied to vector fields in
three dimensions, and defined as $\curl{\phi} = (\partial_2 \phi,
-\partial_1 \phi)^T$ for a scalar function $\phi$ in two dimensions. For
a vector field in two dimensions or a matrix field in three
dimensions, the curl operator produces a matrix field in two or three
dimensions, respectively, by acting row-wise.

Given a vector field $f$ on $\Omega$ representing body
forces, the equations of static elasticity in Hellinger-Reissner form
determine the stress $\sigma$ and the displacement $u$ satisfying the
following constitutive and equilibrium equations respectively,
together with appropriate boundary conditions:
\begin{align}
& A\s = \epsilon(u), \quad \dvr \s = f \,\,\text{ in } \Omega, \label{pde1}\\
& u = g_D \,\,\text{ on } \Gamma_D, \quad
\s\,n = 0 \,\,\text{ on } \Gamma_N, \label{pde2}
\end{align}
where $\epsilon(u) = \frac12(\nabla u + (\nabla u)^T)$ and $n$ is the
outward unit normal vector field on $\partial \O = \G_D \cup \G_N$.
For simplicity we assume that $\meas(\Gamma_D) > 0$, in which case
the problem \eqref{pde1}--\eqref{pde2} has a unique solution.

We will make use of the following standard notations. For a set $G
\subset \Rd$, the $\L2(G)$ inner product and norm are denoted by
$(\cdot,\cdot)_G$ and $\|\cdot\|_G$ respectively, for scalar, vector
and tensor valued functions. For a section of a subdomain boundary $S$
we write $\langle \cdot, \cdot \rangle_{S}$ and $\|\cdot\|_{S}$ for
the $\L2(S)$ inner product (or duality pairing) and norm,
respectively. We omit subscript $G$ if $G = \O$ and $S$ if $S =
\G$. We also denote by $C$ a generic positive constant independent of
the discretization parameters.  We note that, using \eqref{A-defn}, we
have
$$
\inp[A\sigma]{\tau}
= \frac{1}{2\mu}\inp[\sigma]{\tau} - \frac{\lambda}{2\mu(2\lambda + d\mu)}\inp[\tr(\sigma)]{\tr(\tau)},
$$
implying
\begin{equation}
\frac{1}{2\mu+d\lambda}\|\sigma\|^2 \le \inp[A\sigma]{\sigma} \le \frac{1}{2\mu}\|\sigma\|^2. \label{coer-cont}
\end{equation}

We consider the mixed variational formulation for
\eqref{pde1}--\eqref{pde2} with weakly imposed stress symmetry.
Introducing a rotation Lagrange multiplier $\gamma \in \N$ to penalize
the asymmetry of the stress tensor, we obtain: find $(\sigma, u, \g)
\in \X \times V \times \W$ such that
\begin{align}
& \inp[A\s]{\tau} + \inp[u]{\dvr{\tau}} + \inp[\g]{\tau} 
= \gnp[g_D]{\tau\,n}_{\G_D},
&\forall \tau &\in \X, \label{weak1}\\
& \inp[\dvr \sigma]{v} = \inp[f]{v}, &\forall v &\in V, \\
& \inp[\sigma]{\xi} = 0, &\forall \xi &\in \W, \label{weak2}
\end{align}
where
$$ \X = \big\{ \tau\in H(\dvr;\O,\M) : \tau\,n = 0 \text{ on } \G_N  \big\},
\quad V = L^2(\O, \R^d), \quad \W = L^2(\O, \N), $$
with norms
$$ \|\tau\|_{\X} = \left( \|\tau\|^2 + \|\dvr\tau\|^2 \right)^{1/2}, \quad \|v\|_{V} = \|v\|, \quad \|\xi\|_\W = \|\xi\|. $$
It is known \cite{arnold2007mixed} that \eqref{weak1}--\eqref{weak2}
has a unique solution.

\subsection{MFE approximation}

In the first part of the paper we consider a global conforming
shape regular and
quasi-uniform finite element partition $\Th_{h}$ of $\O$. We assume
that $\Th_{h}$ consists of simplices or rectangular elements, but note
that the proposed methods can be extended to other types of elements
for which stable elasticity MFE spaces have been developed, e.g., the
quadrilateral elements in \cite{ArnAwaQiu}. Let
$$
\X_{h}\times V_{h}\times \W_{h} \subset \X \times V \times \W
$$
be any stable triple of spaces for linear elasticity with weakly imposed
stress symmetry, such as the Amara-Thomas \cite{Amara-Thomas}, PEERS \cite{PEERS},
Stenberg \cite{stenberg1988family}, Arnold-Falk-Winther
\cite{arnold2007mixed,ArnAwaQiu,Awanou-rect-weak}, or
Cockburn-Gopalakrishnan-Guzman \cite{cockburn2010new,gopalakrishnan2012second}
families of elements. For all spaces $\dvr \X_{h} = V_{h}$
and there exists a projection operator $\Pi: H^1(\O,\M) \to \X_{h}$,
such that for any $\tau\in H^1(\O,\M)$,
\begin{align}
& \inp[\dvr{(\Pi\, \tau - \tau)}]{v}_{\O} = 0, \quad \forall \, v \in V_{h}, \label{mixed-proj-prop1}\\
& \gnp[(\Pi\, \tau - \tau)\,n]{\chi\,n}_{\dO} = 0, \quad \forall \, \chi \in \X_{h}. \label{mixed-proj-prop2}
\end{align}
The MFE approximation of \eqref{weak1}--\eqref{weak2}
is: find $(\sigma_h, u_h, \g_h) \in \X_h \times V_h \times \W_h$ such
that
\begin{align}
& \inp[A\s_h]{\tau} + \inp[u_h]{\dvr{\tau}} + \inp[\g_h]{\tau} 
= \gnp[g_D]{\tau\,n}_{\G_D},
&\forall \tau &\in \X_h, \label{discr1}\\
& \inp[\dvr \sigma_h]{v} = \inp[f]{v}, &\forall v &\in V_h, \\
& \inp[\sigma_h]{\xi} = 0, &\forall \xi &\in \W. \label{discr2}
\end{align}
The well-posedness of \eqref{discr1}--\eqref{discr2} has been shown in
the above-mentioned references. It was also shown in
\cite{arnold2007mixed,cockburn2010new,gopalakrishnan2012second} that the
following error estimate holds:
\begin{equation}\label{error-est}
\|\s - \s_h\| + \|\P_h u - u_h\| + \|\g - \g_h\| \le
C(\|\s - \Pi\s\| + \|\g - \Rc_h\g\|),
\end{equation}
where $\P_h$ is the $L^2(\O)$-projection onto $V_{h}$ and $\Rc_h$ is the
$L^2(\O)$-projection onto $\W_h$. Later we will also use the restrictions of the
global projections on a subdomain $\O_i$, denoted as $\Pi_i$, $\P_{h,i}$, and
$\Rc_{h,i}$.

\subsection{Domain decomposition formulations}

Let $\O = \cup_{i=1}^n \Omega_i$ be a union of nonoverlapping shape
regular polygonal subdomains. Let $\G_{i,j} = \dO_i \cap \dO_j,\, \G =
\cup_{i,j=1}^n \G_{i,j},$ and $\G_i = \dO_i \cap \G = \dO_i\setminus
\dO$ denote the interior subdomain interfaces. Denote the restrictions of
$\X_h$, $V_h$, and $\W_h$ to $\O_i$ by $\X_{h,i}$, $V_{h,i}$, and $\W_{h,i}$,
respectively. Let $\Th_{h,i,j}$ be a finite element partition
of $\Gamma_{i,j}$ obtained from the trace of $\Th_h$ and let
$\Lambda_{h,i,j} = \X_h\,n$ be the Lagrange multiplier space on $\Th_{h,i,j}$.
Let $\Lambda_h = \bigoplus_{1 \le i,j \le n} \Lambda_{h,i,j}$. We now present
two domain decomposition formulations. The first one uses a displacement
Lagrange multiplier to impose weakly continuity of normal stress.

\medskip
\noindent
\textbf{Method 1:} For $1 \le i \le n$, find
$(\s_{h,i}, u_{h,i}, \g_{h,i},\l_h) \in \X_{h,i}\times V_{h,i}\times \W_{h,i}\times\Lambda_h$ such that
\begin{align}
& \inp[A\sh{i}]{\tau}_{\O_i} + \inp[\uh{i}]{\dvr{\tau}}_{\O_i} 
+ \inp[\gh{i}]{\tau}_{\O_i} && \nonumber \\        
& \qquad\qquad\qquad  = \gnp[\l_h]{\tau\, n_i}_{\G_i} 
+ \gnp[g_D]{\tau\,n_i}_{\partial \O_i \cap \G_D},
&&\forall \tau \in \X_{h,i}, \label{dd1-1}\\
& \inp[\dvr \sh{i}]{v}_{\O_i} = \inp[f]{v}_{\O_i}, &&\forall v \in V_{h,i}, \label{dd1-2} \\
& \inp[\sh{i}]{\xi}_{\O_i} = 0, &&\forall \xi \in \W_{h,i}, \label{dd1-3} \\
& \sum_{i=1}^n \gnp[\sh{i}\,n_i]{\mu}_{\G_i} = 0, &&
\forall \mu \in \Lambda_h, \label{dd1-4}
\end{align}
where $n_i$ is the outward unit normal vector field on $\partial \O_i$. We note
that the subdomain problems in the above method are of Dirichlet type.

The second method uses a normal stress Lagrange multiplier to impose weakly
continuity of displacement. Let
$\X_{h,i}^0 = \{ \t\in\X_{h,i} : \t\,n=0 \mbox{ on } \G \}$ and let $\X_h^\G$
be the complementary subspace:
$$
\X_h = \bigoplus \X_{h,1}^0 \cdots \bigoplus \X_{h,n}^0 \bigoplus \X_h^\G.
$$

\medskip
\noindent
\textbf{Method 2: }
For $1 \le i \le n$, find
$(\s_{h,i}, u_{h,i}, \g_{h,i})
\in \X_{h,i}\times V_{h,i}\times \W_{h,i}$ such that
\begin{align}
& \inp[A\sh{i}]{\tau}_{\O_i} + \inp[\uh{i}]{\dvr{\tau}}_{\O_i} + \inp[\gh{i}]{\tau}_{\O_i} = \gnp[g_D]{\tau\,n_i}_{\partial \O_i \cap \G_D}, &&\forall \tau \in \X_{h,i}^0, \label{dd2-1}\\
& \inp[\dvr \sh{i}]{v}_{\O_i} = \inp[f]{v}_{\O_i}, &&\forall v \in V_{h,i}, \label{dd2-2}\\
& \inp[\sh{i}]{\xi}_{\O_i} = 0, &&\forall \xi \in \W_{h,i}, \label{dd2-3}\\
& \sum_{i=1}^n \sh{i}\, n_i = 0 \quad \mbox{on } \G, \label{dd2-4} \\
& \sum_{i=1}^n \left[ \inp[A\sh{i}]{\tau}_{\O_i} + \inp[\uh{i}]{\dvr{\tau}}_{\O_i} + \inp[\gh{i}]{\tau}_{\O_i} \right] = 0, &&\forall \tau \in \X_{h}^\G.\label{dd2-5}
\end{align}
We note that \eqref{dd2-5} imposes weakly continuity of displacement on the
interface, since taking $\tau \in \X_{h}^\G$ in \eqref{dd2-1}
and summing gives
$$
0 = \sum_{i=1}^n \left[ \inp[A\sh{i}]{\tau}_{\O_i} + \inp[\uh{i}]{\dvr{\tau}}_{\O_i} + \inp[\gh{i}]{\tau}_{\O_i} \right] =
\sum_{i=1}^n \gnp[u_{h,i}]{\t\,n_i}_{\G} \quad \forall \t\in \X_{h}^\G.
$$
It is easy to see that both \eqref{dd1-1}--\eqref{dd1-4}
and \eqref{dd2-1}--\eqref{dd2-5} are equivalent to the global formulation
\eqref{discr1}--\eqref{discr2} with
$(\s_{h}, u_{h}, \g_{h})|_{\O_i} = (\s_{h,i}, u_{h,i}, \g_{h,i})$. In Method 1,
$\l_h$ approximates $u|_{\G}$.

\section{Reduction to an interface problem and 
condition number analysis}
\subsection{Method 1} \label{sec:method1}
To reduce \eqref{dd1-1}--\eqref{dd1-4} to an interface problem for $\l_h$,
we decompose the solution as
\begin{align}
\sh{i} =\ss_{h,i}(\lambda_h) + \bar{\sigma}_{h,i}, && u_{h,i} = \us_{h,i}(\lambda_h) + \ub_{h,i}, && \g_{h,i} = \gs_{h,i}(\lambda_h) + \gb_{h,i}, \label{defs1}
\end{align}
where, for $\lambda_h \in \Lambda_h$, $\left(\sigma^*_i(\lambda_h),
u_i^*(\lambda_h), \g^*_i(\lambda_h)\right) \in \X_{h,i} \times V_{h,i} \times
W_{h,i},\, 1\le i\le n,$ solve
\begin{align}
& \inp[A\sh{i}^*(\lambda_h)]{\tau}_{\O_i} 
+ \inp[\uh{i}^*(\lambda_h)]{\dvr{\tau}}_{\O_i} + \inp[\gh{i}^*(\lambda_h)]{\tau}_{\O_i} 
\nonumber \\
& \qquad\qquad\qquad\qquad
= \gnp[\lambda_h]{\tau\, n_i}_{\G_i}, &\forall \tau & \in \X_{h,i}, \label{star-eqn1}\\
& \inp[\dvr \sh{i}^*(\lambda_h)]{v}_{\O_i} = 0, &\forall v & \in V_{h,i}, \label{star-eqn2}\\
& \inp[\sh{i}^*(\lambda_h)]{\xi}_{\O_i} = 0, &\forall \xi & \in \W_{h,i},
\label{star-eqn3}
\end{align}
and $(\bar{\sigma}_{h,i}, \ub_{h,i}, \gb_{h,i}) \in \X_{h,i}\times V_{h,i}\times \W_{h,i}$
solve
\begin{align}
&\inp[A\bar{\sigma}_{h,i}]{\tau}_{\O_i} + \inp[\ub_{h,i}]{\dvr{\tau}}_{\O_i}
+ \inp[\gb_{h,i}]{\tau}_{\O_i} = \gnp[g_D]{\tau\,n_i}_{(\dO_i\cap\G_D)},
&&\forall \tau \in \X_{h,i}, \label{bar1-1}\\
&\inp[\dvr \bar{\sigma}_{h,i} ]{v}_{\O_i} = \inp[f]{v}_{\O_i},
&&\forall v_{i} \in V_{h,i}, \\
&\inp[\bar{\sigma}_{h,i}]{\xi}_{\O_i} = 0, &&\forall \xi \in \W_{h,i}. \label{bar1-3}
\end{align}
Define the bilinear forms $a_i: \Lambda_{h} \times \Lambda_{h} \to
\mathbb{R}$, $1 \le i \le n$ and $a : \Lambda_h\times \Lambda_h \to
\mathbb{R}$ and the linear functional $g : \Lambda_h \to \mathbb{R}$ by
\begin{align}
&a_i(\lambda_h, \mu) = -\gnp[\sh{i}^*(\lambda_h)\,n_i]{\mu}_{\G_i},
\quad a(\lambda_h,\mu) = \sum_{i=1}^n a_i(\lambda_h, \mu), \label{def-int-bf}\\
&g(\mu) = \sum_{i=1}^n \gnp[\bar{\sigma}_i \,n_i]{\mu}_{\G_i}.
\label{def-int-bf-2}
\end{align}
Using \eqref{dd1-4}, we conclude that the functions satisfying
\eqref{defs1} solve \eqref{dd1-1}--\eqref{dd1-4} if and only if
$\l_h\in\Lambda_h$ solves the interface problem
\begin{align}
a(\l_h,\m) = g(\m) \quad \forall\m\in\Lambda_h. \label{interface-prob}
\end{align}
In the analysis of the interface problem we will utilize the
elliptic projection $\Pit_i: H^1(\O_i,\M) \to \X_{h,i}$ introduced in
\cite{arnold2014mixed}. Given $\s\in \X$ there exists a triple
\\$(\tilde\s_{h,i},\tilde u_{h,i},\tilde \g_{h,i}) \in \X_{h,i}\times V_{h,i} \times \W_{h,i}$
such that
\begin{align}
&\inp[\tilde\s_{h,i}]{\tau}_{\O_i} + \inp[\tilde u_{h,i}]{\dvr{\tau}}_{\O_i}
+ \inp[\tilde\g_{h,i}]{\tau}_{\O_i}
= \inp[\s]{\t}_{\O_i}, &\forall \tau &\in \X^0_{h,i}, \label{proj-1}\\
&\inp[\dvr \tilde\s_{h,i}]{v}_{\O_i} = \inp[\dvr \s]{v}_{\O_i},
&\forall v &\in V_{h,i}, \label{proj-2} \\
&\inp[\tilde\sigma_{h}]{\xi}_{\O_i} = \inp[\s]{\xi}_{\O_i},
&\forall \xi &\in \W_{h,i} \label{proj-3}, \\
&\tilde\s_{h,i} n_i = (\Pi_i\s) n_i \quad \mbox{on } \dO_i. \label{proj-4}
\end{align}
Namely, $(\tilde\s_{h,i},\tilde u_{h,i},\tilde \g_{h,i})$ is a mixed
method approximation of $(\s,0,0)$ based on solving a Neumann
problem. We note that the problem is singular, with the solution
determined up to $(0,\chi,\skew(\nabla \chi))$, $\chi \in
\mathbb{RM}(\O_i)$, where $\mathbb{RM}(\O_i)$ is the space of rigid
body motions in $\O_i$ and $\skew(\tau) = (\tau - \tau^T)/2$ is the
skew-symmetric part of $\tau$. The problem is well posed, since the
data satisfies the compatibility condition
$$
\inp[\dvr\s]{\chi}_{\O_i} - \gnp[(\Pi_i\s) n_i]{\chi}_{\dO_i}
+ \inp[\s]{\skew(\nabla \chi)}_{\O_i} = 0
\quad \forall \chi \in \mathbb{RM}(\O_i),
$$
where we used \eqref{mixed-proj-prop2} on $\dO_i$. We note that the
definition in \cite{arnold2014mixed} is based on a Dirichlet problem,
but it is easy to see that their arguments extend to the Neumann problem.
We now define $\Pit_i\s = \tilde\s_{h,i}$. If $\s \in \X_{h,i}$ we have
$\tilde\s_{h,i}=\s$, $\tilde u_{h,i}=0$, $\tilde\g_{h,i}=0$, so $\Pit$
is a projection. It follows from \eqref{proj-2}--\eqref{proj-4}
and \eqref{mixed-proj-prop2} that for all $\s\in\X,\,\, \xi\in\W_h$, the projection operator $\Pit$ satisfies
\begin{align}
\label{prop-Pit}
\dvr\Pit_i\s &= \P_{h,i}\dvr\s, \quad
\inp[\Pit_i\s]{\xi}_{\O_i} = \inp[\s]{\xi}_{\O_i},
\quad (\Pit_i \s) n_i = \Q_{h,i} (\s n_i),
\end{align}
where $\Q_{h,i}$
is the $L^2(\dO_i)$-projection onto $\X_{h,i} n_i$. Moreover, the error
estimate \eqref{error-est} for the MFE approximation
\eqref{proj-1}--\eqref{proj-3} implies that, see
\cite{arnold2014mixed} for details,
\begin{align}\label{prop-elliptic-proj}
\|\s-\Pit_i\s\|_{\O_i}\le C\|\s-\Pi\s\|_{\O_i},\quad \s\in H^1(\O_i,\M).
\end{align}
We also note that for $\s \in H^{\epsilon}(\O_i,\M) \cap \X_i$,
$0 < \epsilon < 1$, $\Pi_i\s$ is well defined \cite{arbogast2000mixed,mathew1989domain},
it satisfies
\begin{align*}
\|\Pi_i\s\|_{\O_i} \le C\left( \|\s\|_{\epsilon,\O_i}
+ \|\dvr\s\|_{\O_i} \right),
\end{align*}
and, if $\dvr \s = 0$,
\begin{equation}\label{pi-eps}
\|\s - \Pi_i\s\|_{\O_i} \le C h^{\epsilon}\|\s\|_{\epsilon,\O_i}
\end{equation}
Bound \eqref{prop-elliptic-proj} allows us to extend these results to
$\Pit_i \s$:
\begin{align}\label{eps-ineq-pit}
\|\Pit_i\s\|_{\O_i} \le C\left( \|\s\|_{\epsilon,\O_i}
+ \|\dvr\s\|_{\O_i} \right),
\end{align}
and, if $\dvr \s = 0$,
\begin{equation}\label{pit-eps}
\|\s - \Pit_i\s\|_{\O_i} \le C h^{\epsilon}\|\s\|_{\epsilon,\O_i}.
\end{equation}

We are now ready to state and prove the main results for the interface
problem \eqref{interface-prob}.

\begin{lemma}
The interface bilinear form $a(\cdot,\cdot)$ is symmetric and positive
definite over $\Lambda_h$.
\end{lemma}
\begin{proof}
For $\mu \in \l_h$, consider \eqref{star-eqn1} with data $\mu$ and
take $\t = \ss_{h,i}(\lambda_h)$, which implies
\begin{align}\label{represent1}
a(\lambda_h,\mu) = \sum_{i=1}^n \inp[A\ss_{h,i}(\m)]{\ss_{h,i}(\l_h)}_{\O_i},
\end{align}
using \eqref{def-int-bf}, \eqref{star-eqn2} and \eqref{star-eqn3}.
This implies that $a(\cdot,\cdot)$ is symmetric and positive
semi-definite over $\Lambda_h$. We now show that if
$a(\lambda_h,\lambda_h) = 0$, then $\lambda_h = 0$.
Let $\O_i$ be a domain adjacent to $\G_D$, i.e. $\meas(\dO_i\cap \G_D) > 0$.
Let $(\psi_i,\phi_i)$ be the solution of the auxiliary problem
\begin{align}
& A \psi_i = \epsilon(\phi_i), \quad \dvr \psi_i = 0
\quad \text{in }\O_i, \label{aux1}\\
& \phi_i = 0 \quad \text{on } \dO_i\cap\G_D, \\
& \psi_i\,n_i = \begin{cases} 0 & \mbox{on } \dO_i\cap\G_N \\
\lambda_h & \mbox{on } \G_i.
\end{cases} \label{aux2}
\end{align}
Since $\psi_i \in H^{\epsilon}(\O_i,\M) \cap \X_i$ for some $\epsilon > 0$,
see e.g. \cite{grisvard2011elliptic},
$\Pit_i\psi_i$ is well defined and we can take $\t = \Pit_i\psi_i$ in
\eqref{star-eqn1}. Noting that $a(\lambda_h,\lambda_h) = 0$ implies
$\ss_{h,i}(\l_h) = 0$, we have, using \eqref{prop-Pit},
\begin{align}
\gnp[\lambda_h]{\lambda_h}_{\G_i}
&= \gnp[\lambda_h]{(\Pit_i\psi_i)n_i}_{\G_i} \nonumber\\
&= \inp[u^*_{h,i}(\lambda_h)]{\dvr{\Pit_i\psi_i}}_{\O_i}
+ \inp[\gh{i}^*(\lambda_h)]{\Pit_i\psi_i}_{\O_i}
= 0,\label{obs}
\end{align}
which implies $\lambda_h = 0$ on $\G_i$. Next,
consider a domain $\O_j$ adjacent to $\O_i$ such that $\meas(\G_{i,j}) > 0$.
Let $(\psi_j, \phi_j)$ be the solution of
(\ref{aux1})--(\ref{aux2}) modified such that $\phi_j = 0$ on
$\G_{i,j}$. Repeating the above argument implies that
that $\lambda_h = 0$ on $\G_j$. Iterating over all
domains in this fashion allows us to conclude that $\lambda_h = 0$ on
$\G$. Therefore $a(\cdot,\cdot)$ is symmetric and positive definite
over $\Lambda_h$.
\end{proof}

As a consequence of the above lemma, the conjugate gradient (CG)
method can be applied
for solving the interface problem (\ref{interface-prob}). We next proceed
with providing bounds on the bilinear form $a(\cdot,\cdot)$, which can be used
to bound the condition number of the interface problem.

\begin{theorem}\label{thm:condnum}
There exist positive constants $C_0$ and $C_1$ independent of $h$
such that
\begin{align}
\forall \l_h \in \Lambda_h, \quad
C_0 \frac{4\mu^2}{2\mu+d\lambda} \|\lambda_h\|^2_{\G} \le a(\lambda_h,\lambda_h)
\le C_1(2\mu+d\lambda)h^{-1}\|\lambda_h\|^2_{\G}. \label{condnum}
\end{align}
\end{theorem}
\begin{proof}
Using the definition of $a_i(\cdot,\cdot)$ from (\ref{def-int-bf}) we get
\begin{align}\label{upper-bound}
a_i(\lambda_h,\lambda_h)
&= -\gnp[\sigma_{h,i}^*(\lambda_h)\,n_i]{\lambda_h}_{\G_i} \nonumber\\
&\le \|\sigma_{h,i}^*(\lambda_h)\,n_i\|_{\G_i}\|\lambda_h\|_{\G_i}
\le Ch^{-1/2}\|\sigma_{h,i}^*(\lambda_h)\|_{\O_i}\|\lambda_h\|_{\G_i},
\end{align}
where in the last step we used the discrete trace inequality
\begin{align}\label{trace-ineq}
\forall \, \t \in \X_{h,i}, \quad
\|\t\,n_i\|_{\dO_i} \le Ch^{-1/2}\|\t\|_{\O_i},
\end{align}
which follows from a scaling argument. Using \eqref{upper-bound}
together with (\ref{coer-cont}) and \eqref{represent1} we get
\begin{align*}
a_i(\l_h,\l_h) \le C(2\mu+d\lambda)h^{-1}\|\lambda_h\|^2_{\G_i}.
\end{align*}
Summing over the subdomains results in the upper bound in (\ref{condnum}).

To prove the lower bound, we again refer to the solution of the
auxiliary problem (\ref{aux1})--(\ref{aux2}) for a domain
$\O_i$ adjacent to $\G_D$ and take $\t = \Pit_i\psi_i$ in
\eqref{star-eqn1} to obtain
\begin{align*}
\|\lambda_h\|^2_{\G_i} &= \gnp[\lambda_h]{\psi_i\,n_i}_{\G_i}
= \gnp[\lambda_h]{(\Pit\psi_i)n_i}_{\G_i} \\
&= \inp[A\sigma_{h,i}^*(\lambda_h)]{\Pit\psi_i}_{\O_i} + \inp[u_{h,i}^*(\lambda_h)]{\dvr{\Pit\psi_i}}_{\O_i} + \inp[\gh{i}^*(\lambda_h)]{\Pit\psi_i}_{\O_i} \\
&= \inp[A\sigma_{h,i}^*(\lambda)]{\Pit\psi_i}_{\O_i}
\le C\frac{1}{2\mu}\|\sigma_{h,i}^*(\lambda_h)\|_{\O_i}
\,\|\psi_i\|_{\epsilon,\O_i}
\le C\frac{1}{2\mu}\|\sigma_{h,i}^*(\lambda_h)\|_{\O_i} \|\l_h\|_{\G_i},
\end{align*}
where we used \eqref{prop-Pit}, \eqref{eps-ineq-pit}, \eqref{coer-cont}, and
the elliptic regularity \cite{lions2011non,grisvard2011elliptic}
\begin{align}\label{ell-reg}
\|\psi_i\|_{1/2,\O_i} \le C\|\l_h\|_{\G_i}.
\end{align}
Using \eqref{coer-cont} and \eqref{represent1},
we obtain that
\begin{align*}
\|\l_h\|^2_{\G_i} \le C\frac{2\mu+d\lambda}{4\mu^2}a_i(\l_h,\l_h).
\end{align*}
Next, consider a domain $\O_j$ adjacent to $\O_i$ with
$\meas(\G_{i,j}) > 0$. Let $(\psi_j, \phi_j)$ be the solution of
(\ref{aux1})--(\ref{aux2}) modified such that $\phi_j = 0$ on
$\G_{i,j}$. Taking $\t = \Pit_j\psi_j$ in \eqref{star-eqn1} for $\O_j$,
we obtain
\begin{align*}
\|\lambda_h\|^2_{\G_j\setminus \G_{i,j}} & =
\inp[A\sigma_{h,j}^*(\lambda)]{\Pit\psi_j}_{\O_j}
- \gnp[\l_h]{\Pit_j\psi_j \, n_j}_{\G_{i,j}}\\
& \le C\left(\frac{1}{2\mu}\|\sigma_{h,j}^*(\lambda_h)\|_{\O_j}
\|\l_h\|_{\G_j\setminus \G_{i,j}} + \|\l_h\|_{\G_{i,j}}\|\psi_j \, n_j\|_{\G_{i,j}}
\right)\\
& \le C \frac{\sqrt{2\mu + d\l}}{2\mu}
\left(a_j^{1/2}(\l_h,\l_h)
+ a_i^{1/2}(\l_h,\l_h)\right)\|\l_h\|_{\G_j\setminus \G_{i,j}},
\end{align*}
where for the last inequality we used the trace inequality $\|\psi_j
\, n_j\|_{\G_{i,j}} \le C \|\psi_j\|_{1/2,\O_j}$, which follows by
interpolating $\|\psi_j \, n_j\|_{-1/2,\dO_j} \le C
\|\psi_j\|_{H(\dvr;\O_j)} = C \|\psi_j\|_{\O_j}$
\cite{brezzi1991mixed} and $\|\psi_j \, n_j\|_{\e,\dO_j} \le C
\|\psi_j\|_{1/2+\e,\dO_j}$ \cite{grisvard2011elliptic}, together with
the elliptic regularity \eqref{ell-reg}. Iterating over all subdomains
in a similar fashion completes the proof of the lower bound in
\eqref{condnum}.
\end{proof}

\begin{corollary}
Let $A:\Lambda_h\to\Lambda_h$ be such that $\gnp[A \,\lambda]{\mu}_{\G} =
a(\lambda,\mu) \,\, \forall \, \l,\mu \in \Lambda_h$. Then there exists
a positive constant $C$ independent of h such that
$$
\text{cond}(A) \le C \left(\frac{2\m + d \l}{2\m} \right)^2 h^{-1}.
$$
\end{corollary}

\subsection{Method 2}

We introduce the bilinear forms $b_i : \X_{h}^\G\times\X_{h}^\G \to \R$,
$1\le i \le n$, and $b : \X_h^\G\times\X_h^\G \to \R$ by
\begin{align*}
&b_i(\l_h,\m) = \inp[A\ss_{h,i}(\l_h)]{\m}_{\O_i}
+ \inp[\us_{h,i}(\l_h)]{\dvr{\m}}_{\O_i} + \inp[\gs_{h,i}(\l_h)]{\m}_{\O_i}, \\
&b(\l_h,\m) = \sum_{i=1}^n b_i(\l_h,\m),
\end{align*}
where, for a given $\l_h\in\X_{h}^\G$,
$(\ss_{h,i}(\l_h), \us_{h,i}(\l_h), \gs_{h,i}(\l_h)) \in \X_{h,i}\times V_{h,i}\times \W_{h,i}$ solve
\begin{align}
&\inp[A\ss_{h,i}(\l_h)]{\tau}_{\O_i}
+ \inp[\us_{h,i}(\l_h)]{\dvr{\tau}}_{\O_i}
+ \inp[\gs_{h,i}(\l_h)]{\tau}_{\O_i} = 0, &\forall \tau &\in \X_{h,i}^0,
\label{star2-1}\\
&\inp[\dvr \ss_{h,i}(\l_h) ]{v}_{\O_i} = 0, &\forall v &\in V_{h,i},
\label{star2-2} \\
&\inp[\ss_{h,i}(\l_h)]{\xi}_{\O_i} = 0, &\forall \xi &\in \W_{h,i},
\label{star2-3}\\
&\ss_{h,i}(\l_h)\,n_i = \l_h\,n_i \quad \mbox{on } \G_i. \label{star2-4}
\end{align}
Define the linear functional $h:\X_{h}^\G \to \R$ by
\begin{align}
h(\m) = -\sum_{i=1}^n \left[ \inp[A\bar{\sigma}_i]{\m}_{\O_i}
+ \inp[\ub_i]{\dvr{\m}}_{\O_i} + \inp[\gb_i]{\m}_{\O_i} \right],
\end{align}
where $(\bar{\sigma}_i, \ub_i, \gb_i) \in \X_{h,i}^0\times V_{h,i}\times \W_{h,i}$ solve
\begin{align}
&\inp[A\bar{\sigma}_{h,i}]{\tau}_{\O_i} + \inp[\ub_{h,i}]{\dvr{\tau}}_{\O_i}
+ \inp[\gb_{h,i}]{\tau}_{\O_i} = \gnp[g_D]{\tau\,n_i}_{\partial \O_i \cap \G_D},
&\forall \tau &\in \X_{h,i}^0, \label{bar2-1}\\
&\inp[\dvr \bar{\sigma}_{h,i} ]{v}_{\O_i} = \inp[f]{v}_{\O_i},
&\forall v &\in V_{h,i}, \label{bar2-2}\\
&\inp[\bar{\sigma}_{h,i}]{\xi}_{\O_i} = 0, &\forall \xi &\in \W_{h,i}.
\label{bar2-3}
\end{align}
By writing
\begin{align}
\sigma_{h,i} =\ss_{h,i}(\lambda_h) + \bar{\sigma}_{h,i}, && u_{h,i} = \us_{h,i}(\lambda_h) + \ub_{h,i}, && \g_{h,i} = \gs_{h,i}(\lambda_h) + \gb_{h,i} \label{defs2},
\end{align}
it is easy to see that the solution to \eqref{dd2-1}--\eqref{dd2-5}
satisfies the following interface problem: find $\l_h\in\X_{h}^\G$ such that
\begin{align}\label{interface-prob2}
b(\l_h,\m) = h(\m), \quad \forall \m\in\X_h^\G.
\end{align}

\begin{remark}
We note that the Neumann subdomain problems \eqref{star2-1}--\eqref{star2-4} and
\eqref{bar2-1}--\eqref{bar2-3} are singular if $\dO_i \cap \Gamma_D = \emptyset$.
In such case the compatibility conditions for the solvability of
\eqref{star2-1}--\eqref{star2-4} and \eqref{bar2-1}--\eqref{bar2-3} are,
respectively, $\gnp[\lambda_h n_i]{\chi}_{\G_i} = 0$ and
$\inp[f]{\chi}_{\O_i} = 0$ for all $\chi \in \mathbb{RM}(\O_i)$. These can be
guaranteed by employing the one-level FETI method \cite{farhat1991method,Toselli-Widlund}.
This involves solving a coarse space problem, which projects the
interface problem onto a subspace orthogonal to the kernel of the subdomain
operators, see \cite{VasWangYot} for details. In the following we analyze
the interface problem in this subspace, denoted by
$$
\X_{h,0}^\G = \{\mu \in \X_{h}^\G: \gnp[\mu \, n_i]{\chi}_{\G_i} = 0 \,\, \forall \,
\chi \in \mathbb{RM}(\O_i), \forall \, i \mbox{ such that }
\dO_i \cap \Gamma_D = \emptyset \}.
$$
\end{remark}

\begin{lemma}\label{interface-spd2}
The interface bilinear form $b(\cdot,\cdot)$ is symmetric and positive
definite over $\X_{h,0}^\G$.
\end{lemma}
\begin{proof}
We start by showing that
\begin{align}\label{represent2}
b(\l_h,\m) = \sum_{i=1}^n \inp[A\ss_{h,i}(\l_h)_i]{\ss_{h,i}(\m)}_{\O_i}.
\end{align}
To this end, consider the following splitting of $\m$:
$$
\m = \ss_h(\m) + \sum_{i=1}^n \s_{h,i}^0,
$$
where $\ss_h(\m) \big|_{\O_i} = \ss_{h,i}(\m)$ and $\s_{h,i}^0 \in \X_{h,i}^0$.
The the definition of $b_i(\cdot,\cdot)$ reads
\begin{align*}
b_i(\l_h,\m) &= \inp[A\ss_{h,i}(\l_h)]{\ss_{h,i}(\m)}_{\O_i}
+ \inp[\us_{h,i}(\l_h)]{\dvr{\ss_{h,i}(\m)}}_{\O_i}
+ \inp[\gs_{h,i}(\l_h)]{\ss_{h,i}(\m)}_{\O_i} \\
&\quad+ \inp[A\ss_{h,i}(\l_h)]{\s_{h,i}^0}_{\O_i}
+ \inp[\us_{h,i}(\l_h)]{\dvr{\s_{h,i}^0}}_{\O_i}
+ \inp[\gs_{h,i}(\l_h)]{\s_{h,i}^0}_{\O_i} \\
&= \inp[A\ss_{h,i}(\l_h)]{\ss_{h,i}(\m)}_{\O_i},
\end{align*}
using \eqref{star2-1}, \eqref{star2-2} and \eqref{star2-3}. Therefore
\eqref{represent2} holds, which implies that $b(\l_h,\m)$ is symmetric
and positive definite. We next note that, since
$\ss_{h,i}(\l_h) \in H(\dvr,\O_i)$ and $\ss_{h,i}(\l_h)n_i = 0$ on
$\dO_i\setminus \G_i$, then
$\ss_{h,i}(\l_h)n_i = \lambda_h n_i \in H^{-1/2}(\G_i)$ and the normal trace
inequality \cite{galvis2007non} implies
\begin{align}\label{b-spd}
C\|\l_h\,n_i\|^2_{H^{-1/2}(\G_i)} \le \|\ss_{h,i}(\l_h)\|^2_{H(\dvr,\O_i)}
= \|\ss_{h,i}(\l_h)\|^2_{L^2(\O_i)} \le (2\m + d\l)b_i(\l_h,\l_h),
\end{align}
using \eqref{coer-cont} and \eqref{star2-2}. Summing over $\O_i$
proves that $b(\l_h,\l_h)$ is positive definite on $\X_{h,0}^\G$.
\end{proof}
The lemma above shows that the system \eqref{interface-prob2} can be
solved using the CG method. We next prove a bound on $b(\l_h,\l_h)$
that provides an estimate on the condition number of the algebraic
system arising from \eqref{interface-prob2}.
\begin{theorem}
There exist positive constants $c_0$ and
$c_1$ independent of $h$ such that
\begin{align}
\forall \lambda_h \in \X_{h,0}^\G, \quad
c_0\frac{1}{2\m + d\l}h\|\lambda_h\,n\|^2_{\G} \le b(\lambda_h,\lambda_h)
\le c_1 \frac{1}{2\m} \|\lambda_h\,n\|^2_{\G}. \label{condnum2}
\end{align}
\end{theorem}
\begin{proof}
Using \eqref{b-spd} and the inverse inequality \cite{ciarlet2002finite}
we have
\begin{align}
b_i(\l_h,\l_h) \ge C\frac{1}{2\m + d\l}\|\l_h\,n_i\|^2_{H^{-1/2}(\G_i)}
\ge C\frac{1}{2\m + d\l}h\|\l_h\,n_i\|^2_{\G_i},
\end{align}
and the left inequality in \eqref{condnum2} follows from summing over
the subdomains. To show the right inequality, we consider the auxiliary problem
\begin{align*}
& A \psi_i = \epsilon(\phi_i), \quad \dvr \psi_i = 0
\quad \text{in }\O_i, \\
& \phi_i = 0 \quad \text{on } \dO_i\cap\G_D, \\
& \psi_i\,n_i = \begin{cases} 0 & \mbox{on } \dO_i\cap\G_N \\
\lambda_h n_i & \mbox{on } \G_i.
\end{cases}
\end{align*}
Since $\lambda_h \in \X_{h,0}^\G$, the problem is well posed, even
if $\dO_i\cap\G_D = \emptyset$. From elliptic regularity
\cite{lions2011non,grisvard2011elliptic},
$\psi_i \in H^{\epsilon}(\O_i,\M) \cap \X_i$ for some $\epsilon > 0$ and
\begin{align*}
\|\psi_i\|_{\epsilon,\O_i} \le C\|\l_h n_i\|_{\epsilon-1/2,\G_i}.
\end{align*}
We also note that $\ss_{h,i}(\l_h)$ is the MFE approximation of
$\psi_i$, therefore, using \eqref{error-est}, \eqref{pi-eps}, and a
similar approximation property of $\Rc_{h,i}$, the following error
estimate holds:
$$
\|\ss_{h,i}(\l_h) - \psi_i \|_{\O_i} \le C h^{\epsilon} \|\psi_i\|_{\epsilon,\O_i}.
$$
Using the above two bounds, we have
\begin{align*}
\|\ss_{h,i}(\l_h)\|_{\O_i} \le
\|\ss_{h,i}(\l_h) - \psi_i \|_{\O_i} + \|\psi_i\|_{\O_i}
\le C \|\psi_i\|_{\epsilon,\O_i} \le C \|\l_h n_i\|_{\G_i}.
\end{align*}
Squaring the above bound, using \eqref{represent2} and \eqref{coer-cont},
and summing over the subdomains completes the proof of the right inequality in
\eqref{condnum2}.
\end{proof}

\begin{corollary}
Let $B:\X_{h,0}^\G\to\X_{h,0}^\G$ be such that $\gnp[B \,\lambda]{\mu}_{\G} =
b(\lambda,\mu) \,\, \forall \, \l,\mu \in \X_{h,0}^\G$. Then there exists
a positive constant $C$ independent of h such that
$$
\text{cond}(B) \le C \frac{2\m + d \l}{2\m} h^{-1}.
$$
\end{corollary}

\section{A multiscale mortar MFE method on non-matching grids}
\label{sec:mmmfe}
\subsection{Formulation of the method}
In this section we allow for the subdomain grids to be non-matching
across the interfaces and employ coarse scale mortar finite elements
to approximate the displacement and impose weakly the continuity of
normal stress. This can be viewed as a non-matching grid extension of
Method 1. The coarse mortar space leads to a less computationally
expensive interface problem. The subdomains are discretized on the
fine scale, resulting in a multiscale approximation. We focus on the
analysis of the multiscale discretization error.

For the subdomain discretizations, assume that
$\X_{h,i}$, $V_{h,i}$, and $\W_{h,i}$ contain polynomials of degrees
up to $k \ge 1$, $l \ge 0$, and $p \ge 0$, respectively. Let
$$
\X_h = \bigoplus_{1\le i \le n}\X_{h,i}, \quad V_h = \bigoplus_{1\le i \le n}V_{h,i},
\quad \W_h = \bigoplus_{1\le i \le n}\W_{h,i},
$$
noting that the normal traces of stresses in $\X_h$ can be discontinuous across
the interfaces. Let $\T_{H,i,j}$ be a shape regular quasi-uniform simplicial or
quadrilateral finite element partition of $\G_{i,j}$ with maximal
element diameter $H$. Denote by $\Lambda_{H,i,j} \subset
L^2(\Gamma_{i,j})$ the mortar finite element space on $\Gamma_{i,j}$,
containing either continuous or discontinuous piecewise polynomials of
degree $m \ge 0$ on $\T_{H,i,j}$. Let
\begin{align*}
\Lambda_H = \bigoplus_{1\le i,j \le n} \Lambda_{H,i,j}.
\end{align*}
be the mortar finite element space on $\Gamma$. Some additional
restrictions are to be made on the mortar space $\Lambda_h$ in the
forthcoming statements.

The multiscale mortar MFE method reads: find
$(\sh{i},\uh{i},\gh{i},\l_H)\in\X_{h,i}\times V_{h,i}\times \W_{h,i} \times \Lambda_H$ such that, for $1\le i\le n$,
\begin{align}
&\inp[A\sh{i}]{\tau}_{\O_i} + \inp[\uh{i}]{\dvr{\tau}}_{\O_i} + \inp[\gh{i}]{\tau}_{\O_i} \nonumber \\
&\qquad\qquad\qquad= \gnp[\lambda_H]{\tau\, n_i}_{\G_i} + \gnp[g_D]{\tau\,n}_{\partial \O_i \cap \G_D}, &\forall \tau &\in \X_{h,i}, \label{dd3-1}\\
&\inp[\dvr \sh{i}]{v}_{\O_i} = \inp[f]{v}_{\O_i}, &\forall v &\in V_{h,i},\label{dd3-2} \\
&\inp[\sh{i}]{\xi}_{\O_i} = 0, &\forall q &\in \W_{h,i}, \label{dd3-3}\\
&\sum_{i=1}^n \gnp[\sh{i}\,n_i]{\mu}_{\G_i} = 0, &\forall \mu &\in \Lambda_H. \label{dd3-4}
\end{align}
Note that $\lambda_H$ approximates the displacement on $\Gamma$ and the last
equation enforces weakly continuity of normal stress on the interfaces.

\begin{lemma}
Assume that for any $\eta\in\Lambda_H$
\begin{align}\label{lambda-space-cond1}
\Q_{h,i}\eta = 0,\quad 1\le i\le n,\quad \mbox{implies that } \eta=0.
\end{align}
Then there exists a unique solution of \eqref{dd3-1}--\eqref{dd3-3}.
\end{lemma}
\begin{remark}
Condition \eqref{lambda-space-cond1} requires that the mortar
space $\Lambda_H$ cannot be too rich compared to the normal trace of the
stress space. This condition can be easily satisfied in practice, especially
when the mortar space is on a coarse scale.
\end{remark}
\begin{proof}
It suffices to show uniqueness, as \eqref{dd3-1} - \eqref{dd3-4} is a
square linear system. Let $f = 0$ and $g_D = 0$.
Then, by taking $(\t,v,\xi,\m) =
(\s_h,u_h,\g_h,\l_H)$ in \eqref{dd3-1}--\eqref{dd3-4}, we obtain that
$\s_h = 0$. Next, for $1\le i \le n$, let $\overline{u_{h,i}}$ be the
$L^2(\O_i)$-projection of $u_{h,i}$ onto $\RM(\O_i)$ and let
$\overline{\Q_{h,i}\l_H}$ be the $L^2(\G_i)$-projection of $\Q_{h,i}\l_H$
onto $\RM(\O_i)|_{\G_i}$. Consider the auxiliary problem
\begin{align*}
&\psi_i = \epsilon(\phi_i) & &\text{in }\O_i, \\
&\dvr \psi_i = u_{h,i}-\overline{u_{h,i}} & &\text{in } \O_i, \\
&\psi_i\,n_i =
\begin{cases}
-(\Q_{h,i}\l_H - \overline{\Q_{h,i}\l_H}) &\text{on } \G_i \\
0 &\text{on } \dO_i\cap\dO,
\end{cases}
\end{align*}
which is solvable and $\phi$ is determined up to an element of
$\RM(\O_i)$. Now, setting $\t = \Pit_i\psi_i$ in \eqref{dd3-1} and using
\eqref{prop-Pit}, we obtain
\begin{align*}
\inp[u_{h,i}]{u_{h,i} - \overline{u_{h,i}}}_{\O_i}
+ \gnp[\Q_{h,i}\l_H]{\Q_{h,i}\l_H-\overline{\Q_{h,i}\l_H}}_{\G_i} &= 0,
\end{align*}
which implies
$u_{h,i}= \overline{u_{h,i}}$ and $\Q_{h,i}\l_H = \overline{\Q_{h,i}\l_H}$.
Taking $\tau$ to be a symmetric
matrix in \eqref{dd3-1} and integrating by parts gives
$$
-\inp[\epsilon(u_{h,i})]{\tau}_{\O_i}
+ \gnp[u_{h,i} - \l_H]{\tau\, n_i}_{\G_i}
+ \gnp[u_{h,i}]{\tau\, n_i}_{\dO_i\cap\G_D}
= 0.
$$
The first term above is zero, since $u_{h,i} \in \RM(\O_i)$.  Then the last
two terms imply that $u_{h,i}= \Q_{h,i}\l_H$ on $\G_i$ and
$u_{h,i}= 0$ on $\dO_i\cap\G_D$, since
$\RM(\O_i)|_{\dO_i} \in \X_{h,i} n_i$. Using that
$u_{h,i} \in \RM(\O_i)$, this implies that for subdomains $\O_i$
such that $\meas(\dO_i\cap\G_D) > 0$, $u_{h,i} = \Q_{h,i}\l_H = 0$.
Consider any subdomain $\O_j$ such that
$\dO_i\cap\dO_j = \G_{i,j} \neq \emptyset$. Recalling that $k \ge 1$, we have that
for all linear functions $\varphi$ on $\Gamma_{i,j}$,
$$
0 = \gnp[\Q_{h,i}\l_H]{\varphi}_{\G_{i,j}} = \gnp[\l_H]{\varphi}_{\G_{i,j}}
= \gnp[\Q_{h,j}\l_H]{\varphi}_{\G_{i,j}},
$$
which implies that $\Q_{h,j}\l_H = 0$ on $\dO_j$, since
$\Q_{h,j}\l_H \in \RM(\O_j)|_{\dO_j}$. Repeating the above argument
for the rest of the subdomains, we conclude that $\Q_{h,i}\l_H = 0$ and
$u_{h,i} = 0$ for $1\le i \le n$. The hypothesis \eqref{lambda-space-cond1}
implies that $\l_H = 0$. It remains to show that $\g_h = 0$. The stability of
$\X_{h,i}\times V_{h,i} \times \W_{h,i}$ implies an inf-sup condition, which,
along with \eqref{dd3-1}, yields
\begin{align*}
C(\|u_{h,i}\|_{\O_i} + \|\g_{h,i} \|_{\O_i})
&\le \sup_{\t\in\X_{h,i}} \frac{\inp[u_{h,i}]{\dvr\t}_{\O_i}
+ \inp[\g_{h,i}]{\t}_{\O_i}}{\|\t\|_{H(\dvr;\O_i)}} \\
&= \sup_{\t\in\X_{h,i}} \frac{-\inp[A\s_{h,i}]{\t}_{\O_i}
+ \gnp[\l_H]{\t\,n}_{\G_i}}{\|\t\|_{H(\dvr;\O_i)}} = 0,
\end{align*}
implying $\g_h = 0$.
\end{proof}

\subsection{The space of weakly continuous stresses}
We start by introducing some interpolation or projection operators and
discussing their approximation properties. Recall the projection
operators introduced earlier: $\Pi_i$ - the mixed projection operator
onto $\X_{h,i}$, $\Pit_i$ - the elliptic projection operator onto
$\X_{h,i}$, $\P_{h,i}$ - the $L^2(\O_i)$-projection onto $V_{h,i}$,
$\Rc_{h,i}$ - the $L^2(\O_i)$-projection onto $\W_{h,i}$, and
$\Q_{h,i}$ - the $L^2(\O_i)$-projection onto $\X_{h,i} n_i$. In
addition, let $\IHc$ be the Scott-Zhang interpolation operator \cite{Scott-Zhang}
into the space $\Lambda_H^c$, which is the subset of continuous functions in
$\Lambda_H$, and let $\P_H$ be the $\L2(\G)$-projection onto
$\Lambda_H$. Recall that the polynomial degrees in the spaces
$\X_{h,i}$, $V_{h,i}$, $\W_{h,i}$, and $\Lambda_H$ are $k \ge 1$, $l
\ge 0$, $p \ge 0$, and $m \ge 0$, respectively, assuming
for simplicity that the order of approximation is the same on every
subdomain. the projection/interpolation operators have the approximation
properties:
\begin{align}
&\|\eta - \IHc \eta \|_{t,\G_{i,j}} \le CH^{s-t}\|\eta\|_{s,\G_{i,j}}, &&1 \le s \le m+1,\:0\le t\le 1, \label{app-prop0}\\
&\|\eta - \P_H \eta \|_{-t,\G_{i,j}} \le CH^{s+t}\|\eta\|_{s,\G_{i,j}}, &&0 \le s \le m+1,\:0\le t\le 1, \label{app-prop1}\\
&\|w-P_{h,i}w\|_{\O_i} \le Ch^t\|w\|_{t,\O_i}, &&0\le t \le l+1, \label{app-prop2}\\
&\|\dvr(\t-\Pit_i\t)\|_{0,\O_i} \le Ch^t\|\dvr\t\|_{t,\O_i}, &&0\le t \le l+1 \label{app-prop4}\\
&\|\xi-\Rc_{h,i}\|_{\O_i}\le Ch^q\|w\|_{q,\O_i}, &&0\le q \le p+1, \label{app-prop22}\\
&\|\t-\Pit_i\t\|_{\O_i} \le Ch^r\|\t\|_{r,\O_i}, &&1\le r \le k+1, \label{app-prop3}\\
&\|\eta-\Q_{h,i}\eta\|_{-t,\G_{i,j}} \le Ch^{r+t}\|\eta\|_{r,\G_{i,j}}, &&0\le r\le k+1,\: 0\le t\le k+1, \label{app-prop5}\\
&\|(\t-\Pit_i\t)\,n_i\|_{-t,\G_{i,j}} \le Ch^{r+t}\|\t\|_{r,\G_{i,j}}, && 0\le r \le k+1,\: 0\le t\le k+1.\label{app-prop6}
\end{align}
Bound \eqref{app-prop0} can be found in \cite{Scott-Zhang}.
Bounds \eqref{app-prop1}--\eqref{app-prop22} and
\eqref{app-prop5}--\eqref{app-prop6} are well known
$\L2$-projection approximation results \cite{ciarlet2002finite}.
Bound \eqref{app-prop3} follows from \eqref{prop-elliptic-proj} and a similar
bound for $\Pi_i$, which can be found, e.g., in
\cite{brezzi1991mixed,roberts1991mixed}.

We will use the trace inequalities \cite[Theorem 1.5.2.1]{grisvard2011elliptic}
\begin{align}
\|\eta\|_{r,\G_{i,j}} \le C\|\eta\|_{r+1/2,\O_i},\quad r>0
\label{trace-ineq-nonst}
\end{align}
and \cite{brezzi1991mixed,roberts1991mixed}
\begin{align}\label{usefullbound}
\gnp[\eta]{\t\,n}_{\dO_i} \le C\|\eta\|_{1/2,\dO_i}\|\t\|_{H(\dvr;\O_i)}.
\end{align}

We now introduce the space of weakly continuous stresses
with respect to the mortar space,
\begin{align}
\X_{h,0} = \left\{ \t\in\X_h : \sum_{i=1}^n \gnp[\t_in_i]{\m}_{\G_i} = 0
\quad\forall\m\in\Lambda_H \right\}.
\end{align}
Then the mixed method \eqref{dd3-1}--\eqref{dd3-4} is equivalent to:
find $(\s_h,u_h,\g_h)\in\X_{h,0}\times V_h\times \W_h$ such that
\begin{align}
&\inp[A\s_h]{\tau}_{\O_i} + \sum_{i=1}^n\inp[u_h]{\dvr{\tau}}_{\O_i} + \sum_{i=1}^n\inp[\g_h]{\tau}_{\O_i} = \gnp[g_D]{\tau\,n}_{\G_D}, &\forall \tau &\in \X_{h,0}, \label{dd3-1-w}\\
&\sum_{i=1}^n\inp[\dvr \s_h]{v}_{\O_i} = \inp[f]{v}, &\forall v &\in V_{h},\label{dd3-2-w} \\
&\sum_{i=1}^n\inp[\s_h]{\xi}_{\O_i} = 0, &\forall q &\in \W_{h} \label{dd3-3-w}.
\end{align}
We note that the above system will be used only for the purpose of the analysis.
We next construct a projection operator $\Pit_0$ onto $\X_{h,0}$ with optimal
approximation properties.
The construction follows closely the approach in \cite{arbogast2000mixed,APWY}. Define
\begin{align*}
\X_h\,n = \big\{ (\eta_L,\eta_R) &\in L^2(\G,\R^d)\times L^2(\G,\R^d) : \\
&\eta_L\big|_{\G_{i,j}}\in\X_{h,i}\,n_i, \, \eta_R\big|_{\G_{i,j}}\in\X_{h,j}\,n_j
\quad \forall \, 1\le i< j \le n \big\}
\end{align*}
and
\begin{align*}
\X_{h,0}\,n = \big\{ (&\eta_L,\eta_R) \in L^2(\G,\R^d)\times L^2(\G,\R^d) : \exists \t\in\X_{h,0}\mbox{ such that } \\
&\eta_L\big|_{\G_{i,j}} = \t_i n_i \mbox{ and }
\eta_R\big|_{\G_{i,j}} = \t_j n_j \quad \forall \, 1\le i< j \le n \big\}.
\end{align*}
For any $\eta = (\eta_L,\eta_R)\in \left(L^2(\G,\R^d)\right)^2$ we
write $\eta\big|_{\G_{i,j}} = (\eta_i,\eta_j)$, $1\le i<j\le
n$. Define the $L^2$-projection $\Q_{h,0}: \left(L^2(\G,\R^d)\right)^2 \to
\X_{h,0}\,n$ such that, for any $\eta \in \left(L^2(\G,\R^d)\right)^2$,
\begin{align}
\sum_{i=1}^n\gnp[\eta_i - (\Q_{h,0}\eta)_i]{\phi_i}_{\G_i} = 0,
\quad \forall \, \phi\in \X_{h,0}\,n. \label{qh0-def}
\end{align}
\begin{lemma}
Assume that \eqref{lambda-space-cond1} holds. Then, for any
$\eta \in \left(L^2(\G,\R^d)\right)^2$, there exists
$\lambda_H\in\Lambda_H$ such that on $\G_{i,j}$, $1 \le i\le j\le n$,
\begin{align}
&\Q_{h,i}\l_H = \Q_{h,i}\eta_i - (\Q_{h,0}\eta)_i, \label{qhi-prop-1}\\
&\Q_{h,j}\l_H = \Q_{h,j}\eta_j - (\Q_{h,0}\eta)_j, \label{qhi-prop-2}\\
&\gnp[\lambda_H]{\chi}_{\G_{i,j}} =
\frac12\gnp[\eta_i+\eta_j]{\chi}_{\G_{i,j}},  \,\, \forall \, \chi \in
\RM(\O_i \cup\O_j)|_{\G_{i,j}}.
\label{lh-eta-prop}
\end{align}
\end{lemma}
\begin{proof}
The proof is given in \cite[Lemma 3.1]{arbogast2000mixed} with a straightforward modification
to show \eqref{lh-eta-prop} for $\chi \in \RM(\O_i \cup\O_j)|_{\G_{i,j}}$, rather
than for constants.
\end{proof}

The next lemma shows that, under a relatively mild assumption on the mortar
space $\Lambda_H$, $\Q_{h,0}$ has optimal approximation properties.
\begin{lemma} \label{lambda-space-cond2}
Assume that there exists a constant $C$, independent of $h$ and $H$, such that
\begin{align}
\|\m\|_{\G_{i,j}} \le C(\|\Q_{h,i}\m\|_{\G_{i,j}} + \|\Q_{h,j}\m\|_{\G_{i,j}}) \quad
\forall \m\in\Lambda_H, \quad 1\le i< j\le n \label{lambda-cond-eq}.
\end{align}
Then for any $\eta \in \left(L^2(\G,\R^d)\right)^2$ such that
$\eta\big|_{\G_{i,j}} = (\eta_i,-\eta_i)$,
there exists a constant $C$, independent of $h$ and $H$ such that
\begin{align}\label{Qh0-bound}
\begin{aligned}
&\left( \sum_{1\le i< j \le n} \| \Q_{h,i}\eta_i-(\Q_{h,0}\eta)_i \|^2_{-s,\G_{i,j}} \right)^{1/2}\le C\sum_{1\le i<j \le n} h^{r}H^{s} \|\eta_i\|_{r,\G_{i,j}}, \\
&\qquad\qquad
0\le r\le k+1, \: 0\le s\le k+1.
\end{aligned}
\end{align}
\end{lemma}
\begin{proof}
The proof is given in \cite[Lemma 3.2]{arbogast2000mixed} with a straightforward modification
for the two scales $h$ and $H$.
\end{proof}

\begin{remark}
The condition \eqref{lambda-cond-eq} is related to
\eqref{lambda-space-cond1} and it requires that the mortar space
$\Lambda_H$ is controlled by its projections onto the normal traces of
stress spaces with a constant independent of the mesh size. It can be satisfied
for fairly general mesh configurations, see
\cite{arbogast2000mixed,APWY,pencheva2003balancing}.
\end{remark}

We are now ready to construct the projection operator onto $\X_{h,0}$.

\begin{lemma}\label{proj-lemma}
Under assumption \eqref{lambda-cond-eq}, there exists a projection operator
$\Pit_0: H^{1/2+\epsilon}(\O,\M)\cap \X \to \X_{h,0}$ such that
\begin{align}
& \inp[\dvr(\Pit_0\t - \t)]{v}_{\O_i} = 0, &&v \in V_{h,i}, \,\, 1 \le i \le n, \label{pit0-prop-1}\\
& \inp[\Pit_0 \t - \t]{\xi} = 0, && \xi \in \W_h , \label{pit0-prop-2}\\
& \|\Pit_0\t\| \le C (\|\t\|_{1/2 + \e} + \|\dvr \t\|), &&\label{pit0-cont}\\
& \|\Pit_0\t - \Pit\t\| \le Ch^{r}H^{1/2}\|\t\|_{r+1/2},
&& 0 < r\le k+1, \label{app-prop-pit0-1} \\
&\|\Pit_0\t - \t\| \le C \left(h^{t}\|\t\|_{t} +
h^r H^{1/2}\|\t\|_{r+1/2}\right), && 1 \le t \le k+1, \,\,
0 < r\le k+1. \label{app-prop-pit0-2}
\end{align}
\end{lemma}
\begin{proof}
For any $\t\in H^{1/2+\epsilon}(\O,\M)\cap \X$ define
$$ \Pit_0\t\big|_{\O_i} = \Pit_i(\t + \delta\t_i), $$
where $\delta\t_i$ solves
\begin{align}
&\delta\t_i = \epsilon(\phi_i) & &\text{in }\O_i \label{mortar-aux-correct-1}\\
&\dvr \delta\t_i = 0 & &\text{in } \O_i, \label{mortar-aux-correct-2} \\
&\delta\t_i\,n_i = \begin{cases} 0,\hspace{3.5cm}
&\mbox{on } \dO_i\cap\dO \\
-\Q_{h,i}\t\,n_i + (\Q_{h,0}\t\,n)_i, & \mbox{on } \G_i,
\end{cases} \label{mortar-aux-correct-3}
\end{align}
wherein, on any $\G_{i,j}$, $\t\,n\big|_{\G_{i,j}} =
(\t\,n_i,\t\,n_j)$. Note that the assumed regularity of $\tau$ and the trace
inequality \eqref{trace-ineq-nonst} imply that
$\t\,n_i = -\t\,n_j \in L^2(\G_{i,j},\R^d)$,
so Lemma~\ref{lambda-space-cond2} holds for $\t\,n\big|_{\G_{i,j}}$.
The Neumann problems
\eqref{mortar-aux-correct-1}--\eqref{mortar-aux-correct-3} are
well-posed, since $\forall \chi \in \RM(\O_i)|_{\G_{i,j}}$ by \eqref{qhi-prop-1} and \eqref{lh-eta-prop} there holds
$$
\gnp[\Q_{h,i}\t\,n_i - (\Q_{h,0}\t\,n)_{i}]{\chi}_{\G_{i,j}} =
\gnp[\Q_{h,i}\l_H]{\chi}_{\G_{i,j}}
= \frac12\gnp[\t\,n_i + \t\,n_j]{\chi}_{\G_{i,j}} = 0.
$$
Also, note that the piecewise polynomial Neumann data are in
$H^{\epsilon}(\dO_i)$, so $\delta\t_i\in H^{\epsilon + 1/2}(\O_i,\M)$;
thus, $\Pit_i$ can be applied to $\delta\t_i$, see \eqref{eps-ineq-pit}.
We have by \eqref{prop-Pit} that
\begin{align*}
\sum_{i=1}^n\gnp[(\Pit_0\t)\,n_i]{\m}_{\G_i} =
\sum_{i=1}^n \gnp[(\Q_{h,0}\t\,n)_i]{\m}_{\G_i} = 0, \quad \forall\,\m\in\Lambda_H,
\end{align*}
therefore $\Pit_0\t \in \X_{h,0}$. Also, \eqref{prop-Pit} implies
\begin{align*}
\inp[\dvr\Pit_0\t]{v}_{\O_i} = \inp[\dvr\Pit_i\t]{v}_{\O_i} + \inp[\dvr\Pit_i\delta\t_i]{v}_{\O_i} = \inp[\dvr\t]{v}_{\O_i},\quad \forall\,v\in V_{h,i},
\end{align*}
so \eqref{pit0-prop-1} holds. In addition, \eqref{pit0-prop-2} holds
due to \eqref{prop-Pit} and the fact that $\delta\t_i$ is a symmetric
matrix.  It remains to study the approximation properties of
$\Pit_0$. Since $ \Pit_0\t - \t = \Pit_i\t - \t + \Pit_i\delta\t_i$ on
$\O_i$, and using \eqref{app-prop3}, it suffices to bound only the
correction term. By the elliptic regularity of
\eqref{mortar-aux-correct-1}-\eqref{mortar-aux-correct-3}
\cite{lions2011non,grisvard2011elliptic}, for any $0\le t \le 1/2$,
\begin{align}
\|\delta\t_i\|_{t,\O_i} \le
\sum_{j}\| \Q_{h,i}\t\,n_i - (\Q_{h,0}\t\,n)_i \|_{t-1/2,\G_{i,j}}.
\end{align}
We then have, using \eqref{pit-eps},
\begin{align*}
&\| \Pit_i\delta\t_i \|_{0,\O_i}
\le \| \Pit_i\delta\t_i - \delta\t_i \|_{0,\O_i} + \|\delta\t_i\|_{0,\O_i}
\le Ch^{1/2}\|\delta\t_i\|_{1/2,\O_i} + \|\delta\t_i\|_{0,\O_i} \\
&\quad\le C\sum_{j}\left[ h^{1/2}\|\Q_{h,i}\t\,n_i - (\Q_{h,0}\t\,n)_i\|_{0,\G_{i,j}} + \| \Q_{h,i}\t\,n_i - (\Q_{h,0}\t\,n)_i \|_{-1/2,\G_{i,j}}\right],
\end{align*}
which, together with \eqref{Qh0-bound} and \eqref{trace-ineq-nonst},
implies \eqref{app-prop-pit0-1}. Then \eqref{pit0-cont} follows from
\eqref{eps-ineq-pit} and \eqref{app-prop-pit0-2} follows from \eqref{app-prop3}.
\end{proof}

\subsection{Optimal convergence for the stress}
We start by noting that, assuming that the solution $u$ of
\eqref{weak1}--\eqref{weak2} belongs to $H^1(\O)$, integration by parts
in the second term in \eqref{weak1} implies that
$$
\inp[u]{\dvr{\tau}} = \sum_{i=1}^n\left( \inp[u]{\dvr{\tau}}_{\O_i}
- \gnp[u]{\tau\,n_i}_{\G_i} \right).
$$
Using the above and subtracting \eqref{dd3-1-w}--\eqref{dd3-3-w} from
\eqref{weak1}--\eqref{weak2} gives the error equations
\begin{align}
&\inp[A(\s-\s_h)]{\tau}_{\O} + \sum_{i=1}^n\big[\inp[u-u_h]{\dvr{\tau}}_{\O_i}
+ \inp[\g-\g_h]{\tau}_{\O_i}\big] \nonumber\\
&\qquad\qquad\qquad\qquad\quad\:\:= \sum_{i=1}^n \gnp[u]{\tau\,n_i}_{\G_i},
&\forall \tau &\in \X_{h,0}, \label{err-eq-1}\\
&\sum_{i=1}^n\inp[\dvr (\s-\s_h)]{v}_{\O_i}
= 0, &\forall v &\in V_{h}\label{err-eq-2}, \\
&\sum_{i=1}^n\inp[\s-\s_h]{\xi}_{\O_i} = 0,
&\forall q &\in \W_{h} \label{err-eq-3}.
\end{align}
It follows from \eqref{err-eq-2} and \eqref{pit0-prop-1} that
\begin{align}\label{div-0}
\dvr(\Pit_0\s - \s_h) = 0 \quad \mbox{in } \O_i.
\end{align}
Similarly, \eqref{err-eq-3} and \eqref{pit0-prop-2} imply
\begin{align*}
\inp[\Pit_0\s -\s_h]{\xi} = 0, \quad \xi\in\W_h.
\end{align*}
Taking $\t = \Pit_0\s - \s_h$ in \eqref{err-eq-1} and using
that $\sum_{i} \gnp[\IHc v]{\tau\,n_i}_{\G_i} = 0$ for any $\t \in
\X_{h,0}$, we obtain
\begin{align*}
& \inp[A(\Pit_0\s-\s_h)]{\Pit_0\s - \s_h}
= \inp[A(\Pit_0\s-\s)]{\Pit_0\s - \s_h} \\
&\qquad\qquad + \sum_{i=1}^n \inp[\Rc_h\g-\g]{\Pit_0\s - \s_h}_{\O_i}
+ \sum_{i=1}^n\gnp[\IHc u-u]{(\Pit_0\s - \s_h)\,n_i}_{\G_i} \\
&\quad \le
C\left( \|\Pit_0\s-\s\|\|\Pit_0\s-\s_h\|
+ \|\Rc_h\g-\g\|\|\Pit_0\s-\s_h\| \right. \\
&\qquad \qquad
+ \left. \sum_{i=1}^n\|E_i(\IHc u-u)\|_{1/2,\dO_i}\|(\Pit_0\s-\s_h)\|_{H(\dvr;\O_i)}
\right) \\
&\quad
\le C\left(h^t\|\s\|_{t} + h^{r}H^{1/2}\|\s\|_{r+1/2} +
h^q\|\g\|_{q} + H^{s-1/2}\| u \|_{s+1/2}\right) \|\Pit_0\s - \s_h\|, \\
&\qquad\qquad
1 \le t \le k+1, \, 0 \le r \le k+1, \, 0 \le q \le p+1,  1 \le s \le m+1,
\end{align*}
where $E_i(\IHc u-u)$ is a continuous extension by zero to $\dO_i$ and
we have used the Cauchy-Schwarz inequality, \eqref{usefullbound},
\eqref{app-prop-pit0-2}, \eqref{app-prop22}, \eqref{app-prop0},
and \eqref{trace-ineq-nonst}. The above inequality, together with
\eqref{app-prop-pit0-2}, \eqref{div-0}, and \eqref{app-prop4}, results in
the following theorem.

\begin{theorem}\label{stress-err-estimate}
For the stress $\s_h$ of the mortar mixed finite element method
\eqref{dd3-1}-\eqref{dd3-4}, if \eqref{lambda-cond-eq} holds, then
there exists a positive constant $C$ independent of $h$ and $H$ such that
\begin{align*}
&\|\s-\s_h\| \le C \left(h^t\|\s\|_{t} + h^{r}H^{1/2}\|\s\|_{r+1/2} +
h^q\|\g\|_{q} + H^{s-1/2}\| u \|_{s+1/2}\right),\\
&\qquad\qquad
1 \le t \le k+1, \, 0 < r \le k+1, \, 0 \le q \le p+1, \,  1 \le s \le m+1,\\
&\|\dvr(\s-\s_h)\|_{\O_i} \le Ch^r\|\dvr\s\|_{r,\O_i},\quad 0 \le r \le l+1.
\end{align*}
\end{theorem}
\begin{remark}\label{r1}
The above result implies that for sufficiently regular solution,
$\|\s - \s_h\| = \mathcal{O}(h^{k+1} + h^{p+1} + H^{m+1/2})$.
The mortar polynomial degree $m$ and the coarse scale $H$ can be
chosen to balance the error terms, resulting in a fine scale
convergence. Since in all cases $p \le k$, the last two error terms
are of the lowest order and balancing them results in the choice $H
=\mathcal{O}(h^\frac{p+1}{m+1/2})$. For example, for the lowest order
Arnold-Falk-Winther space on simplices \cite{arnold2007mixed} and its
extensions to rectangles in two and three dimensions
\cite{Awanou-rect-weak} or quadrilaterals \cite{ArnAwaQiu},
$\X_{h,i}\times V_{h,i} \times \W_{h,i} = \mathcal{BDM}_1\times
\mathcal{P}_0 \times \mathcal{P}_0$, so $k=1$ and $l=p=0$. In this
case, taking $m=2$ and the asymptotic scaling $H
=\mathcal{O}(h^{2/5})$ provides optimal convergence rate
$\mathcal{O}(h)$.  Similarly, for the lowest order
Gopalakrishnan-Guzman space on simplices \cite{gopalakrishnan2012second}
or the modified
Arnold-Falk-Winther space on rectangles with continuous $Q_1$
rotations \cite{MSMFEM-2}, $k=1$, $l=0$, and $p=1$. In this case,
taking $m=2$ and the asymptotic scaling $H=\mathcal{O}(h^{4/5})$
or $m=3$ and $H=\mathcal{O}(h^{4/7})$
provides optimal convergence rate $\mathcal{O}(h^2)$.
\end{remark}

\subsection{Convergence for the displacement}
On a single domain, the error estimate for the displacement and the
rotation follows from an inf-sup condition. For the mortar method, we
would need an inf-sup condition for the space of weakly continuous
stresses $\X_{h,0}$. This can be approached by finding a global stress
function with specified divergence and asymmetry and applying the
projection operator $\Pit_0$. Unfortunately, the regularity of the
global stress function, which can be constructed by solving two
divergence problems, is only $H(\dvr;\O)$, which is not sufficient to
apply $\Pit_0$. For this reason, we split the analysis in three parts.
First, we construct a weakly continuous symmetric stress function with
specified divergence to control the displacement and show both optimal
convergence and superconvergence. In the second step we estimate the
error in the mortar displacement by utilizing the properties of the
interface operator established in the earlier domain decomposition
sections. Finally we construct on each subdomain a divergence-free
stress function with specified asymmetry to bound the error in the
rotation in terms of the error in stress and mortar displacement.

\subsubsection{Optimal convergence for the displacement}
Let $\phi$ be the solution of the problem
\begin{align}
&\dvr \left(A^{-1}\e(\phi)\right) = (\P_hu-u_h) & &\text{in } \O, \label{u-uh-1}\\
&\phi = 0 & &\text{on } \G_D, \label{u-uh-2}\\
&A^{-1}\e(\phi)n = 0 & &\text{on } \G_N. \label{u-uh-3}
\end{align}
Since $\O$ is polygonal and $\P_hu-u_h \in L^2(\O)$, the problem is
$H^{1+r}$-regular for a suitable $r > 1/2$ \cite{Dauge} and
$\|\phi\|_{1+r} \le C \|\P_hu-u_h\|$. Let $\t = \Pit_0 A^{-1}\e(\phi)$,
which is well defined, since $A^{-1}\e(\phi) \in H^r(\O)$. Note that
\eqref{pit0-prop-1} implies that $\dvr \t = \P_hu-u_h$. Also,
\eqref{pit0-cont} implies that $\|\t\| \le C(\P_hu-u_h)$.  Taking this $\t$
as the test function in the error equation \eqref{err-eq-1} gives
\begin{align*}
\|\P_h u - u_h\|^2  & = - \inp[A(\s-\s_h)]{\t} +
\sum_{i=1}^n \gnp[u - \IHc u]{\t\,n}_{\G_i} \\
& \le C \left( \|\s-\s_h\|\|\t\| + \sum_{i=1}^n \|E_i(u - \IHc u)\|_{1/2,\dO_i}
\|\t\|_{H(\dvr;\O_i)} \right) \\
& \leq C \left(\|\s-\s_h\|
+ \sum_{i=1}^n \|E_i(u - \IHc u)\|_{1/2,\dO_i} \right) \|\P_h u - u_h\|,
\end{align*}
which, together with Theorem~\ref{stress-err-estimate}, \eqref{app-prop0}, and
\eqref{app-prop2}, implies the following theorem.
\begin{theorem}\label{thm-u-optimal}
For the displacement $u_h$ of the mortar mixed method
\eqref{dd3-1}--\eqref{dd3-4}, if \eqref{lambda-cond-eq} holds, then
there exists a positive constant $C$ independent of $h$ and $H$ such that
\begin{align}
&\|\P_h u-u_h\| \le C \left(h^t\|\s\|_{t} + h^{r}H^{1/2}\|\s\|_{r+1/2} +
h^q\|\g\|_{q} + H^{s-1/2}\| u \|_{s+1/2}\right), \label{Phu-uh}\\
&\|u-u_h\| \le C \left(h^t\|\s\|_{t} + h^{r}H^{1/2}\|\s\|_{r+1/2} +
h^q\|\g\|_{q} + H^{s-1/2}\| u \|_{s+1/2} + h^{r_u}\|u\|_{r_u}\right),\\
&\qquad
1 \le t \le k+1, \, 0 < r \le k+1, \, 0 \le q \le p+1, \,  1 \le s \le m+1,
\, 0 \le r_u \le l+1. \nonumber
\end{align}
\end{theorem}
\begin{remark}
The above result shows that $\|\P_h u-u_h\|$ is of the same order as
$\|\s - \s_h\|$ and it does not depend on the approximation order of $V_h$.
\end{remark}

\subsubsection{Superconvergence for the displacement}
We present a duality argument to obtain a superconvergence estimate
for the displacement. We utilize again the auxiliary problem
\eqref{u-uh-1}--\eqref{u-uh-3}, but this time we assume that the problem is
$H^2$-regular, see e.g. \cite{grisvard2011elliptic} for sufficient conditions:
\begin{align}
\|\phi\|_2 \le C\|\P_h u - u_h\|. \label{displ-err-0}
\end{align}
Taking $\t = \Pit_0 A^{-1}\e(\phi)$ in \eqref{err-eq-1}, we get
\begin{align}
\|\P_h u - u_h \|^2
=-\sum_{i=1}^n \left[\inp[A(\s-\s_h)]{\Pit_0 A^{-1}\e(\phi)}_{\O_i} - \gnp[u-\P_H u]{\Pit_0 A^{-1}\e(\phi)\, n_i}_{\G_i}\right]
\label{err-eq-displ}.
\end{align}
Noting that $\inp[\s - \s_h]{\e(\phi)} =
\inp[\s - \s_h]{\nabla \phi - \skew(\nabla \phi)}$, 
we manipulate the first term on the right as follows,
\begin{align}
\sum_{i=1}^n & \inp[A(\s-\s_h)]{\Pit_0 A^{-1}\e(\phi)}_{\O_i} \nonumber\\
& = \sum_{i=1}^n \left[\inp[A(\s-\s_h)]{\Pit_0 A^{-1}\e(\phi) -  A^{-1}\e(\phi)}_{\O_i} + \inp[A(\s-\s_h)]{A^{-1}\e(\phi)}_{\O_i} \right] \nonumber\\
&= \sum_{i=1}^n \bigg[\inp[A(\s-\s_h)]{\Pit_0 A^{-1}\e(\phi) - A^{-1}\e(\phi)}_{\O_i} - \inp[\dvr(\s-\s_h)]{\phi - P_h\phi}_{\O_i}   \nonumber\\
&\qquad\qquad + \gnp[(\s-\s_h)n_i]{\phi - \IHc\phi}_{\G_i}
- \inp[\s - \s_h]{\skew(\nabla \phi - \Rc_h\nabla\phi)}_{\O_i} \bigg] \nonumber \\
&\le C\sum_{i=1}^n \bigg[ (\sqrt{hH}+h)\|\s-\s_h\|_{\O_i}
+ h\|\dvr(\s-\s_h)\|_{\O_i} \nonumber\\
&\qquad\qquad+ H\|\s-\s_h\|_{H(\dvr;\O_i)} \bigg]\|\phi\|_{2,\O_i},\label{displ-err-1}
\end{align}
where we used \eqref{app-prop-pit0-2}, \eqref{app-prop2},
\eqref{app-prop0}, and \eqref{app-prop22} for the last inequality with
$C = C(\max_i\|A^{-1}\|_{1,\infty,\O_i})$.
Next, for the second term on the right in \eqref{err-eq-displ} we have
\begin{align}
&\gnp[u-\P_H u]{\Pit_0 A^{-1}\e(\phi)\, n_i}_{\G_i} \nonumber\\
&\quad = \gnp[u-\P_H u]{\left(\Pit_0 A^{-1}\e(\phi)
- \Pit_iA^{-1}\e(\phi)\right)n_i}_{\G_i} \nonumber \\
&\qquad\qquad +\gnp[u-\P_H u]{\left(\Pit_i A^{-1}\e(\phi) - A^{-1}\e(\phi)\right)n_i
+ A^{-1}\e(\phi) n_i} \nonumber\\
&\quad \le \sum_{j}\|u-\P_Hu\|_{\G_{i,j}}
\bigg[ \| \left(\Pit_0 A^{-1}\e(\phi)
- \Pit_iA^{-1}\e(\phi)\right) n_i \|_{\G_{i,j}} \nonumber \\
&\qquad\qquad+ \|\left( \Pit_i A^{-1}\e(\phi) - A^{-1}\e(\phi) \right)n_i\|_{\G_{i,j}} \bigg] \nonumber \\
&\qquad\qquad + \sum_{j} \|u-\P_Hu\|_{-1/2,\G_{i,j}}\|A^{-1}\e(\phi)\,n_i\|_{1/2,\G_{i,j}} \nonumber\\
&\quad\le CH^{s+1/2}\|u\|_{s+1/2,\O_i}\|\phi\|_{2,\O_i}, \quad 0< s\le m+1, \label{displ-err-2}
\end{align}
where we used \eqref{app-prop1}, \eqref{app-prop6},
\eqref{trace-ineq}, and \eqref{app-prop-pit0-1} for the last inequality.
A combination of \eqref{displ-err-0}--\eqref{displ-err-2},
and Theorem \ref{stress-err-estimate} gives the following theorem.
\begin{theorem}
Assume $H^2$-regularity of the problem on $\Omega$
and that \eqref{lambda-cond-eq} holds.
Then there exists a positive constant $C$, independent of $h$ and $H$ such that
\begin{align*}
& \|\P_h u -u_h\| \le C \bigg(h^t H\|\s\|_{t} + h^{r}H^{3/2}\|\s\|_{r+1/2} +
h^q H\|\g\|_{q} \\
&\qquad\qquad \qquad\qquad \qquad\qquad \qquad\qquad + H^{s+1/2}\| u \|_{s+1/2} + h^{r_u} H\|\dvr\s\|_{r_u}
\bigg),\\
&\qquad\qquad
1 \le t \le k+1, \, 0 < r \le k+1, \, 0 \le q \le p+1, \,  1 \le s \le m+1, \,
0 \le r_u \le l+1.
\end{align*}
\end{theorem}
\begin{remark}\label{r2}
The result shows that $\|\P_h u-u_h\| = \mathcal{O}(H(h^{k+1} +
h^{p+1} + h^{l+1} + H^{m+1/2}))$, which is of order $H$ higher that
$\|\s-\s_h\|_{H(\dvr;\O_i)}$. Similar to Remark~\ref{r1}, the error
terms can be balanced to obtain fine scale convergence. For spaces
with optimal stress convergence, $l \le p \le k$, so balancing the
last two terms results in the choice $H
=\mathcal{O}(h^\frac{l+1}{m+1/2})$. For the lowest order spaces in
\cite{arnold2007mixed,Awanou-rect-weak,ArnAwaQiu} with $k=1$ and
$l=p=0$, taking $m=2$ and the asymptotic scaling $H
=\mathcal{O}(h^{2/5})$ provides superconvergence rate
$\mathcal{O}(h^{7/5})$.  We further note that the above result is not
useful for spaces with $l = p-1$, in which case the bound
\eqref{Phu-uh} from Theorem~\ref{thm-u-optimal}, which does not depend
on $l$, provides a better rate.
\end{remark}

\subsection{Convergence for the mortar displacement}
Recall the interface bilinear form $a(\cdot,\cdot): L^2(\G)\times
L^2(\G) \to \R$ introduced in \eqref{def-int-bf} and its
characterization \eqref{represent1},
$a(\lambda,\mu) = \sum_{i=1}^n \inp[A\ss_{h,i}(\m)]{\ss_{h,i}(\l)}_{\O_i}$.
Denote by $\|\cdot\|_{a}$ the seminorm induced by
$a(\cdot,\cdot)$ on $L^2(\G)$, i.e.,
$$
\|\m\|_{a} = a(\m,\m)^{1/2}, \quad \m\in L^2(\G).
$$
\begin{theorem}
For the mortar displacement $\l_H$ of the mixed method
\eqref{dd3-1}--\eqref{dd3-4}, if \eqref{lambda-cond-eq} holds, then
there exists a positive constant $C$, independent of $h$ and $H$, such
that
\begin{align}
& \|u-\l_H\|_{a} \le C \left(h^t\|\s\|_{t} + h^{r}H^{1/2}\|\s\|_{r+1/2} +
h^q\|\g\|_{q} + H^{s-1/2}\| u \|_{s+1/2}\right), \label{mortar-estimate-1} \\
&\qquad\qquad
1 \le t \le k+1, \, 0 < r \le k+1, \, 0 \le q \le p+1, \,  1 \le s \le m+1.
\nonumber
\end{align}
\end{theorem}
\begin{proof}
The characterization \eqref{represent1} implies that
\begin{align}
\| u-\l_H \|_{a} \le C \| \s^*_h(u) - \s^*_h(\l_H) \|. \label{mortar-err-0}
\end{align}
Define, for $\m\in L^2(\G)$,
$$ \s_h(\m) = \s^*_h(\m)+\bar{\s}_h, \quad u_h(\m) = u^*_h(\m)+\bar{u}_h,
\quad \g_h(\m) = \g^*_h(\m)+\bar{\g}_h.
$$
Recalling \eqref{star-eqn1}--\eqref{star-eqn3} and
\eqref{bar1-1}--\eqref{bar1-3}, we
note that $(\s_h(\m),u_h(\m),\g_h(\m)) \in \X_h\times
V_h\times\W_h$ satisfy, for $1\le i \le n$,
\begin{align}
&\inp[A\s(\m)]{\tau}_{\O_i} + \inp[u_h(\m)]{\dvr{\tau}}_{\O_i} + \inp[\g_h(\m)]{\tau}_{\O_i} \nonumber\\ 
&\qquad\qquad\qquad\:\:\:\,= \gnp[g]{\t\,n}_{\partial \O_i \cap \G_D} + \gnp[\m]{\tau\, n_i}_{\G_i} &\forall \tau &\in \X_{h,i}, \label{mortar-err-eqn1}\\
&\inp[\dvr \s_h(\m)]{v}_{\O_i} = (f,v)_{\O_i} &\forall v &\in V_{h,i}, \label{mortar-err-eqn2}\\
&\inp[\s_h(\m)]{\xi}_{\O_i} = 0 &\forall \xi &\in \W_{h,i}\label{mortar-err-eqn3}.
\end{align}
We note that $(\s_h(\l_H), u_h(\l_H), \g_h(\l_H)) = (\s_h, u_h, \g_h)$
and that $(\s_h(u), u_h(u), \g_h(u))$ is the MFE approximation of the
true solution $(\s,u,\g)$ on each subdomain $\O_i$ with specified
boundary condition $u$ on $\Gamma_i$. We then have
\begin{align}
\| \s^*_h(u) - \s^*_h(\l_H) \| = \| \s_h(u) - \s_h(\l_H) \|
= \| \s_h(u) - \s_h \|
\le \| \s_h(u) - \s \| + \| \s - \s_h \| \label{mortar-err-1}.
\end{align}
The assertion of the theorem \eqref{mortar-estimate-1} follows from
\eqref{mortar-err-0}, \eqref{mortar-err-1}, Theorem
\ref{stress-err-estimate}, and the standard mixed method estimate
\eqref{error-est} for \eqref{mortar-err-eqn1}--\eqref{mortar-err-eqn3}.
\end{proof}

\subsection{Convergence for the rotation}
We first note that the result of Theorem~\ref{thm:condnum} holds
in the case of non-matching grids. In particular, it is easy to check that
its proof can be extended to this case, assuming that on each $\G_{i,j}$,
$C_1 \|\Q_{h,i}\m\|_{\G_{i,j}} \le \|\Q_{h,j}\m\|_{\G_{i,j}}
\le C_2 \|\Q_{h,i}\m\|_{\G_{i,j}}$ for all $\m \in \Lambda_H$.
It was shown in \cite{pencheva2003balancing} that this norm equivalence holds
for very general grid configurations. Therefore \eqref{condnum} implies that
$\|\cdot\|_a$ is a norm on $\Lambda_H$.

The stability of the subdomain MFE spaces $\X_{h,i}\times V_{h,i}
\times \W_{h,i}$ implies a subdomain inf-sup condition: there exists a
positive constant $\beta$ independent of
$h$ and $H$ such that, for all $v\in V_{h,i},\, \xi\in \W_{h,i}$,
\begin{align}
\sup_{0 \neq \t\in\X_{h,i} } \frac{\inp[\dvr \t]{v}_{\O_i}
+ \inp[\t]{\xi}_{\O_i} }{\|\t\|_{H(\dvr;\O_i,\M)} }
\ge \beta\left(\|v\|_{\O_i}+\|\xi\|_{\O_i}\right).
\end{align}
Then, using the error equation obtained by subtracting
\eqref{dd3-1} from \eqref{weak1}, we obtain
\begin{align*}
\|\Rc_h\g - \g_h\|_{\O_i} &\le C
\sup_{0 \neq \t\in\X_{h,i} } \frac{\inp[\dvr \t]{\P_h u - u_h}_{\O_i}
+ \inp[\t]{\Rc_h\g - \g_h}_{\O_i} }{\|\t\|_{H(\dvr;\O_i,\M)} } \\
& \le C \sup_{0 \neq \t\in\X_{h,i} }\frac{-\inp[A(\s-\s_h)]{\t}_{\O_i} +
\gnp[u - \l_H]{\t \, n_i}}{\|\t\|_{H(\dvr;\O_i,\M)}}\\
& \le C (\|\s-\s_h\|_{\O_i} + h^{-1/2}\|u - \l_H\|_{\G_i}),
\end{align*}
using the discrete trace inequality \eqref{trace-ineq} in the last inequality.
Summing over the subdomains results in the following theorem.

\begin{theorem}
For the rotation $\g_h$ of the mixed method
\eqref{dd3-1}--\eqref{dd3-4}, if \eqref{lambda-cond-eq} holds, then
there exists a positive constant $C$, independent of $h$ and $H$, such
that
$$
\|\Rc_h\g - \g_h\| \le C(\|\s-\s_h\| + h^{-1/2}\|u - \l_H\|_{\G}).
$$
\end{theorem}
\begin{remark}
The above result, combined with \eqref{condnum}, implies convergence
for the rotation reduced by $\mathcal{O}(h^{-1/2})$ compared to the
other variables, which is suboptimal. Since $\|\cdot\|_a$
is equivalent to a discrete $H^{1/2}(\G)$-norm, see
\cite{pencheva2003balancing}, one expects that
$\|u - \l_H\|_{\G} \le C h^{1/2} \|u - \l_H\|_{a}$, which is indeed observed
in the numerical experiments, and results in optimal convergence
for the rotation.
\end{remark}

\subsection{Multiscale stress basis implementation}
The algebraic system resulting from the multiscale mortar MFE method
\eqref{dd3-1}--\eqref{dd3-4} can be solved by reducing it to an
interface problem similar to \eqref{interface-prob}, as discussed in
Section~\ref{sec:method1}. The solution of the interface problem by
the CG method requires solving subdomain problems on each
iteration. The choice of a coarse mortar space $\Lambda_H$ results in
an interface problem of smaller dimension, which is less expensive to
solve. Nevertheless, the computational cost may be significant if many
CG iterations are needed for convergence.  Alternatively, following
the idea of a multiscale flux basis for the mortar mixed finite
element method for the Darcy problem
\cite{ganis2009implementation,wheeler2012multiscale}, we introduce a
multiscale stress basis.  This basis can be computed before the start
of the interface iteration and requires solving a fixed number of
Dirichlet subdomain problems, equal to the number of mortar degrees of
freedom per subdomain. Afterwards, an inexpensive linear combination
of the multiscale stress basis functions can replace the subdomain
solves during the interface iteration. Since this implementation
requires a relatively small fixed number of local fine scale solves,
it makes the cost of the method comparable to other multiscale
methods, see e.g. \cite{Efe-Gal-Hou} and references therein.

Let $A_H: \Lambda_H \to \Lambda_H$ be an interface operator such that
$\gnp[A_{H}\l]{\mu}_{\Gamma} = a(\l, \mu)$, $\forall \, \l, \, \mu \in \Lambda_H$.
Then the interface problem \eqref{interface-prob} can be rewritten as
$A_H \l_H = g_H$. We note that $A_H \l_H = \sum_{i=1}^n A_{H,i}\l_{H,i}$, where
$A_{H,i}: \Lambda_{H,i} \to \Lambda_{H,i}$ satisfies
$$
\gnp[A_{H,i}\l_{H,i}]{\mu}_{\Gamma_i} =
-\gnp[\sh{i}^*(\lambda_{H,i})n_i]{\mu}_{\G_i} \, \forall \, \mu \in \Lambda_{H,i}.
$$
Let $\Q_{h,i}: \Lambda_{H,i} \to \X_{h,i} n_i$ be the $L^2(\dO_i)$-projection
from the mortar space onto the normal trace of the subdomain velocity and let
$\Q_{h,i}^T: \X_{h,i} n_i \to \Lambda_{H,i}$ be the $L^2(\dO_i)$-projection
from the normal velocity trace onto the mortar space. Then the above implies that
$$
A_{H,i}\l_{H,i} = - \Q_{h,i}^T \sh{i}^*(\lambda_{H,i})n_i.
$$
We now describe the computation of the multiscale stress basis and its
use for computing the action of the interface operator
$A_{H,i}\l_{H,i}$.  Let $\{\phi^{(k)}_{H,i}\}_{k=1}^{N_{H,i}}$ denote
the basis functions of the mortar space $\Lambda_{H,i}$, where
$N_{H,i}$ is the number of mortar degrees of freedom on subdomain
$\O_i$. Then, for $\l_{H,i}\in \Lambda_{H,i}$ we have
$$
\l_{H,i} = \sum_{k=1}^{N_{H,i}} \l_{H,i}^{(k)} \phi^{(k)}_{H,i}.
$$
The computation of multiscale stress basis function $\psi^{(k)}_{H,i}
= A_{H,i}\phi^{(k)}_{H,i}$ is as follows.

\begin{algorithm}
\caption{Compute multiscale basis}
\label{alg:msb}
\begin{algorithmic}
\FOR{$k=1,\dots,N_{H,i}$}
\STATE{1. Project $\phi^{(k)}_{H,i}$ onto the subdomain boundary:
     $ \eta_{i}^{(k)} = \Q_{h,i}\phi^{(k)}_{H,i}$}
\STATE{2. Solve subdomain problem \eqref{star-eqn1}--\eqref{star-eqn3} with
     Dirichlet data $\eta_{i}^{(k)}$ for $\sh{i}^*(\eta_i^{(k)})$}
\STATE{3. Project the boundary normal stress onto the mortar space:\\
     $\qquad \psi^{(k)}_{H,i} = -(\Q_{h,i})^T \sh{i}^*(\eta_i^{(k)})\,n_i$}
\ENDFOR
\end{algorithmic}
\end{algorithm}

%
%
%
Once the multiscale stress basis is computed, the action of interface
operator $A_{H,i}$ involves only a simple linear combination of the
multiscale basis functions:
$$ A_{H,i}\l_{H,i} = A_{H,i} \left( \sum_{k=1}^{N_{H,i}}
\l_{H,i}^{(k)} \phi^{(k)}_{H,i} \right) = \sum_{k=1}^{N_{H,i}}
\l_{H,i}^{(k)} A_{H,i} \phi^{(k)}_{H,i} = \sum_{k=1}^{N_{H,i}}
\l_{H,i}^{(k)} \psi^{(k)}_{H,i}.
$$

\section{Numerical results}
In this section, we provide several numerical tests confirming the
theoretical convergence rates and illustrating the behavior of Method
1 on non-matching grids, testing both the conditioning of the
interface problem studied in Section~\ref{sec:method1} and the
convergence of the numerical errors of the multiscale mortar method
studied in Section~\ref{sec:mmmfe}.  The computational domain for all
examples is a unit hypercube partitioned with rectangular elements.
For simplicity, Dirichlet boundary conditions are specified on
the entire boundary in all examples. In 3
dimensions we employ the $\BDM_1 \times \mathcal{Q}_0 \times
\mathcal{Q}_0$ triple of elements proposed by Awanou
\cite{Awanou-rect-weak}, which are the rectangular analogues of the
lowest order Arnold-Falk-Winther simplicial elements
\cite{arnold2007mixed}. In 2 dimensions we use $\BDM_1 \times \mathcal{Q}_0
\times \mathcal{Q}_1^{cts}$, a modified triple of elements with
continuous $\Q_1$ space for rotation introduced in
\cite{MSMFEM-2}. This choice is of interest, since it allows for local
elimination of stress and rotation via the use of trapezoidal
quadrature rules, resulting in an efficient cell-centered scheme for
the displacement \cite{MSMFEM-2}.

We use the Method 1, with a displacement Lagrange multiplier, for all
tests. The CG method is employed for solving the symmetric and
positive definite interface problems. It is known
\cite{kelley1995iterative} that the number of iterations required for
the convergence of the CG method is $\mathcal{O}(\sqrt{\kappa})$,
where $\kappa$ is the condition number of the interface system.
According to the theory in Section~\ref{sec:method1}, $\kappa =
\mathcal{O}(h^{-1})$, hence the expected growth rate of the number of
iterations is $\mathcal{O}(h^{-1/2})$. We set the tolerance for the CG
method to be $\epsilon = 10^{-14}$ for all test cases and use the zero
initial guess for the interface data, i.e. $\lambda_H = 0$. We used deal.II
finite element library \cite{dealii} for the implementation of the method.

The convergence rates are established by running each test case on a
sequence of refined grids. The coarsest non-matching multiblock grid
consists of $2\times 2$ and $3\times 3$ subdomain grids in a
checkerboard fashion. The mortar grids on the coarsest level have only
one element per interface, i.e. $H = \frac12$.  In 2 dimensions, with
$\BDM_1 \times \mathcal{Q}_0 \times \mathcal{Q}_1^{cts}$, we have
$k=1$, $p = 0$, and $l = 1$. We test quadratic and cubic mortars.
According to Remark~\ref{r1}, $m=2$ and $H=\mathcal{O}(h^{4/5})$ or
$m=3$ and $H=\mathcal{O}(h^{4/7})$ should result in $\mathcal{O}(h^2)$
convergence. In the numerical test we take $H=2h$ for $m=2$ and
$H=h^{1/2}$ for $m=3$, which are easier to do in practice.  In 3
dimensions, with $\BDM_1 \times \mathcal{Q}_0 \times \mathcal{Q}_0$,
we have $k=1$, $p = l = 0$. We test linear mortars, $m=1$. From
Remark~\ref{r1}, the choice $H=\mathcal{O}(h^{2/3})$ should result in
$\mathcal{O}(h)$ convergence. In the numerical test we take
$H=2h$. The theoretically predicted convergence rates for these
choices of finite elements and subdomain and mortar grids
are shown in Table \ref{tab:t0}.
\begingroup
\def\arraystretch{1.2}
\begin{table}[H]
\scriptsize
\begin{center}
\begin{tabular}{c|c|c|c|c|c|c|c}
     \hline
     \multicolumn{8}{c}{$\BDM_1 \times \mathcal{Q}_0 \times \mathcal{Q}_1^{cts}$ ($k=1,\, l=0,\, p=1$) \mbox{ in 2 dimensions}} \\ \hline
     $m$
     & $H$
     & $ \|\sigma - \sigma_h\| $
     & $ \|\dvr(\sigma - \sigma_h)\| $
     & $ \|u - u_h\|$
     & $ \|\P_h u - u_h\| $
     & $ \|\g - \g_h\| $
     & $ \|u - \lambda_H\|_{a}$ \\
     \hline
     2 & $2h$      & 2 & 1 & 1 & 2 & 2 & 2 \\
     3 & $h^{1/2}$ & 2 & 1 & 1 & 2 & 2 & 2 \\\hline

     \multicolumn{8}{c}{$\BDM_1 \times \mathcal{Q}_0 \times \mathcal{Q}_0$ ($k=1,\, l=0,\, p=0$) \mbox{ in 3 dimensions}} \\ \hline
     $m$
     & $H$
     & $ \|\sigma - \sigma_h\| $
     & $ \|\dvr(\sigma - \sigma_h)\| $
     & $ \|u - u_h\|$
     & $ \|\P_h u - u_h\| $
     & $ \|\g - \g_h\| $
     & $ \|u - \lambda_H\|_{a}$ \\
     \hline
     1 & $2h$      & 1 & 1 & 1 & 2 & 1 & 1 \\
     \hline

\end{tabular}
\end{center}
\caption{Theoretical convergence rates for the choices of
finite elements and mortars in the numerical tests.} \label{tab:t0}
\end{table}
\endgroup

In the first three examples we test the convergence rates and the
condition number of the interface operator. The error $\|\P_h u -
u_h\|$ is approximated by the discrete $L^2$-norms computed by the
midpoint rule on $\mathcal{T}_h$, which is known to be
$\mathcal{O}(h^2)$-close to $\|\P_h u-u_h\|$.  The mortar displacement
error $ \|u - \lambda_H\|_{a}$ is computed in accordance with the
definition of the interface bilinear form $a(\cdot,\cdot)$.  In all
cases we observe that the rates of convergence agree with the
theoretically predicted ones. Also, in all cases the number of
CG iterations grows with rate $\mathcal{O}(h^{-1/2})$, confirming
the theoretical condition number $\kappa = \mathcal{O}(h^{-1})$.

\subsection{Example 1}
In the first example we solve a two-dimensional problem with a known
analytical solution
\begin{align*}
u = \begin{pmatrix} x^3y^4 + x^2 + \sin(xy)\cos(y) \\
x^4y^3 + y^2 + \cos(xy)\sin(x) \end{pmatrix}.
\end{align*}
The Poisson's ratio is $\nu = 0.2$ and the Young's modulus
is $E = \sin(3\pi x)\sin(3\pi y) + 5$, with the Lam\'{e} parameters determined by
\begin{align*}
\lambda = \frac{E\nu}{(1-\nu)(1-2\nu)}, \quad \mu = \frac{E}{2(1+2\nu)}.
\end{align*}
Relative errors, convergence rates, and number of interface
iterations are provided in Tables \ref{tab:t1} and
\ref{tab:t2}. The computed solution is plotted in Figure~\ref{fig:1}.
\setlength\tabcolsep{2.5pt}
\begingroup
\def\arraystretch{1.1}
\begin{table}[H]
\scriptsize
\begin{center}
\begin{tabular}{c|cc|cc|cc|cc|cc|cc|cc}
     \hline
     & \multicolumn{2}{c|}{$\|\sigma - \sigma_h\| $ }  & \multicolumn{2}{c|}{$ \|\dvr(\sigma - \sigma_h)\| $} & \multicolumn{2}{c|}{$ \|u - u_h\|$} & \multicolumn{2}{c|}{$ \|\P_hu - u_h\| $}  &
     \multicolumn{2}{c|}{$\|\g - \g_h\| $} & \multicolumn{2}{c|}{$ \|u - \lambda_H\|_{a}$} & \multicolumn{2}{c}{CG iter.}
     \\
     $h$	&	error	&        rate   &	error	&	rate	&	error	&	rate	&	error	&	rate
     &	error	&	rate	&	error	&	rate	&	\#	&	rate             \\
     \hline
     1/4	&	2.02E-1	&       -       &	5.64E-1	&	-	&	4.57E-1	&	-	&	2.54E-1	&	-
     &	4.08E-1	&	-	&	5.01E-1	&	-	&	24	&	-			\\
     1/8	&	5.43E-2	&       1.9     &	2.98E-1	&	0.9	&	2.12E-1	&	1.1	&	7.14E-2	&	1.8
     &	1.04E-1	&	2.0	&	1.33E-1	&	1.9	&	33	&	-0.4			\\
     1/16	&	1.37E-2	&       2.0     &	1.51E-1	&	1.0	&	1.04E-1	&	1.0	&	1.84E-2	&	2.0
     &	2.60E-2	&	2.0	&	3.25E-2	&	2.0	&	48	&	-0.5			\\
     1/32	&	3.42E-3	&       2.0     &	7.58E-2	&	1.0	&	5.15E-2	&	1.0	&	4.63E-3	&	2.0
     &	6.47E-3	&	2.0	&	7.83E-3	&	2.1	&	63	&	-0.5			\\
     1/64	&	8.53E-4	&       2.0     &	3.79E-2	&	1.0	&	2.57E-2	&	1.0	&	1.16E-3	&	2.0
     &	1.61E-3	&	2.0	&	1.88E-3	&	2.1	&	96	&	-0.5			\\
     1/128	&	2.13E-4	&       2.0     &	1.90E-2	&	1.0	&	1.28E-2	&	1.0	&	2.90E-4	&	2.0
     &	4.02E-4	&	2.0	&	4.55E-4	&	2.1	&	136	&	-0.6			\\
     1/256	&	5.33E-5	&       2.0     &	9.48E-3	&	1.0	&	6.42E-3	&	1.0	&	7.25E-5	&	2.0
     &	1.00E-4	&	2.0	&	1.10E-4	&	2.0	&	194	&	-0.5			\\	\hline
\end{tabular}
\end{center}
\caption{Numerical errors, convergence rates, and number of CG iterations with
discontinuous quadratic mortars ($m=2$) for Example 1.} \label{tab:t1}
\end{table}
\endgroup
\begingroup
\def\arraystretch{1.1}
\begin{table}[H]
\scriptsize
\begin{center}
\begin{tabular}{c|cc|cc|cc|cc|cc|cc|cc}
     \hline
     & \multicolumn{2}{c|}{$\|\sigma - \sigma_h\| $ }  & \multicolumn{2}{c|}{$ \|\dvr(\sigma - \sigma_h)\| $} & \multicolumn{2}{c|}{$ \|u - u_h\|$} & \multicolumn{2}{c|}{$ \|\P_hu - u_h\| $}  &
     \multicolumn{2}{c|}{$\|\g - \g_h\| $} & \multicolumn{2}{c|}{$ \|u - \lambda_H\|_{a}$} & \multicolumn{2}{c}{CG iter.}
     \\
     $h$	&	error	&        rate   &	error	&	rate	&	error	&	rate	&	error	&	rate
     &	error	&	rate	&	error	&	rate	&	\#	&	rate             \\
     \hline
     1/4	&	4.05E-2	&	-	&	3.75E-1	&	-	&	1.36E-1	&	-	&	1.09E-2	&	-
     &	1.79E-1	&	-	&	1.99E-2	&	-	&	26	&	-		\\
     1/16	&	3.35E-3	&	1.8	&	1.11E-1	&	0.9	&	3.41E-2	&	1.0	&	9.13E-4	&	1.8
     &	1.06E-2	&	2.0	&	9.42E-4	&	2.2	&	46	&	-0.4		\\
     1/64	&	2.14E-4	&	2.0	&	2.80E-2	&	1.0	&	8.53E-3	&	1.0	&	5.84E-5	&	2.0
     &	6.74E-4	&	2.0	&	4.97E-5	&	2.1	&	78	&	-0.4		\\
     1/256	&	1.34E-5	&	2.0	&	7.01E-3	&	1.0	&	2.13E-3	&	1.0	&	3.62E-6	&	2.0
     &	4.19E-5	&	2.0	&	2.63E-6	&	2.1	&	124	&	-0.3		\\	\hline
\end{tabular}
\end{center}
\caption{Numerical errors, convergence rates, and number of CG iterations with discontinuous cubic mortars ($m=3$) for Example 1.} \label{tab:t2}
\end{table}
\endgroup
\begin{figure}[ht!]
\centering
\begin{subfigure}[b]{0.24\textwidth}
\includegraphics[width=\textwidth]{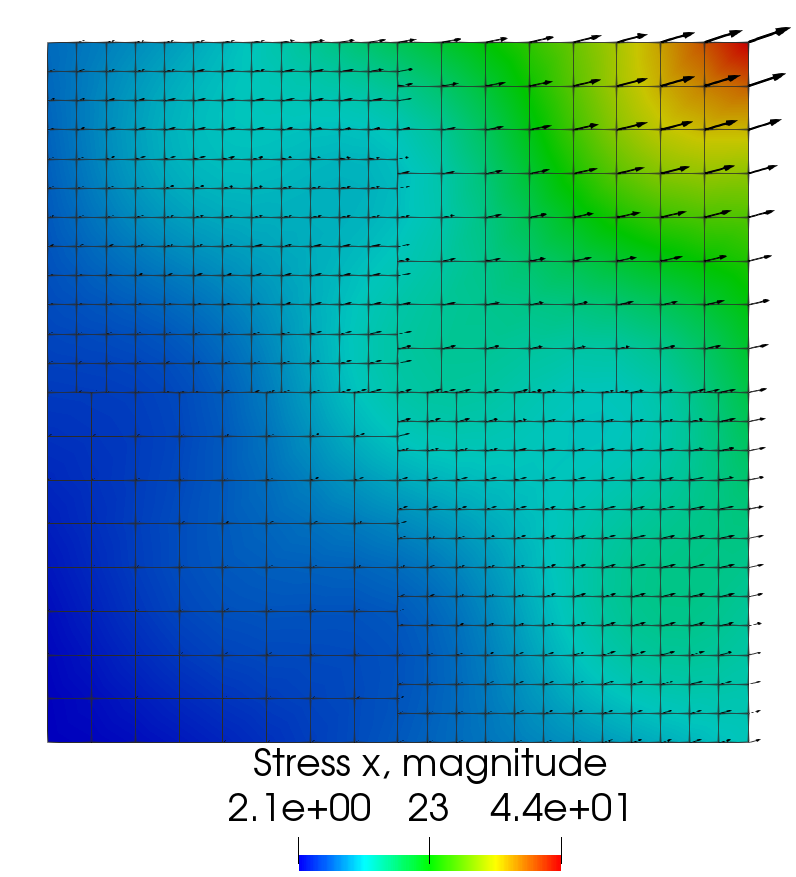}
\label{fig:1_1}
\end{subfigure}
\begin{subfigure}[b]{0.24\textwidth}
\includegraphics[width=\textwidth]{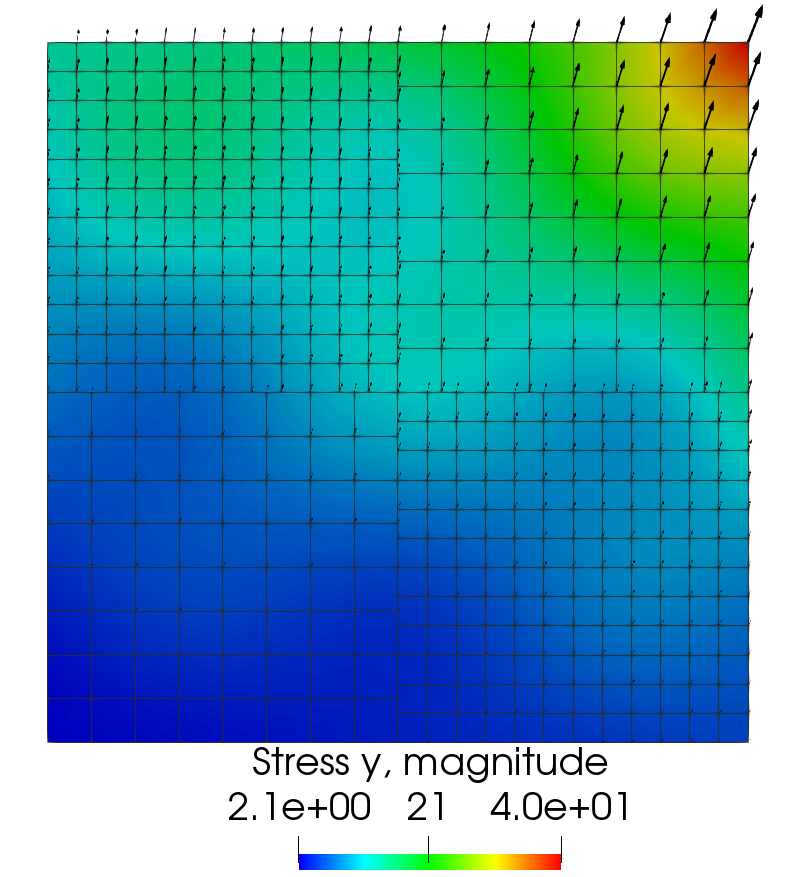}
\label{fig:1_2}
\end{subfigure}
\begin{subfigure}[b]{0.24\textwidth}
\includegraphics[width=\textwidth]{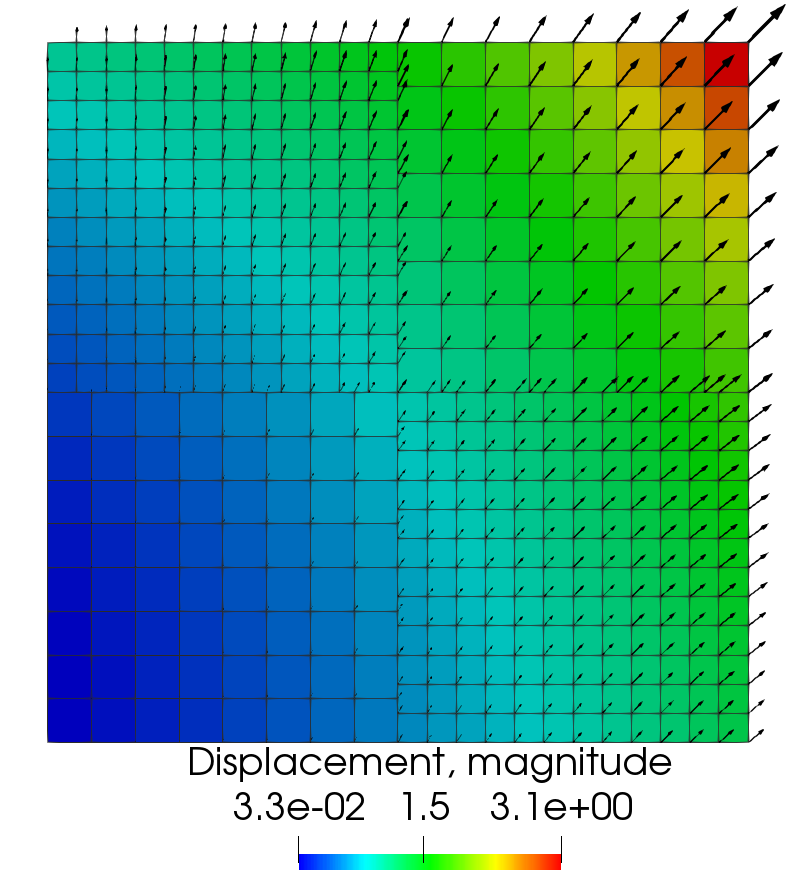}
\label{fig:1_3}
\end{subfigure}
\begin{subfigure}[b]{0.24\textwidth}
\includegraphics[width=\textwidth]{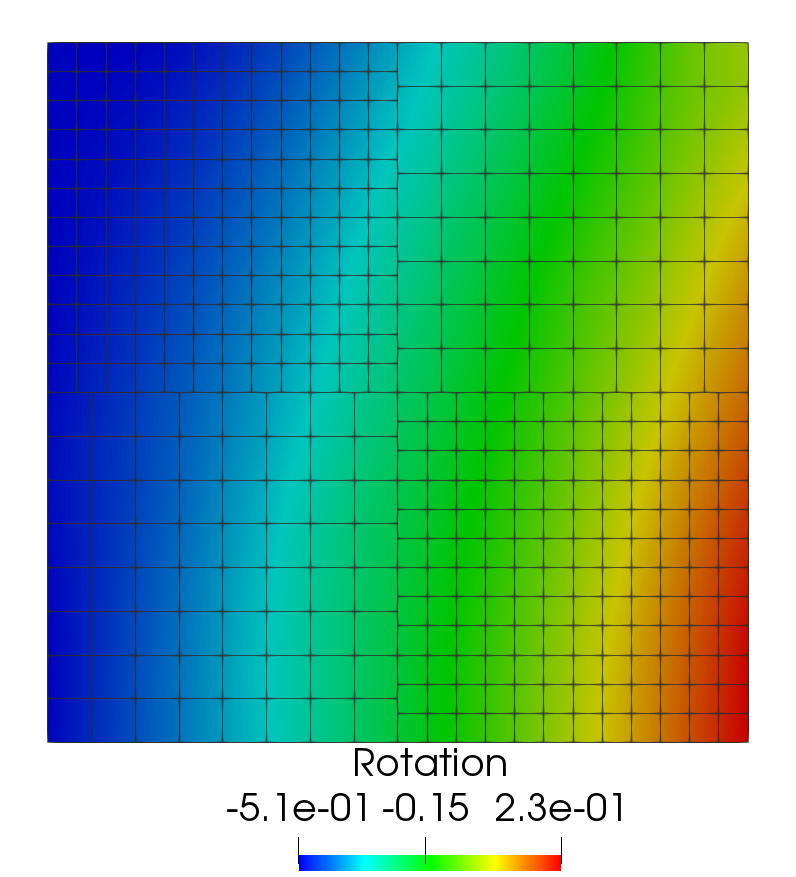}
\label{fig:1_4}
\end{subfigure}
\caption{Computed solution for Example 1, $h=1/16$.}\label{fig:1}
\end{figure}
\subsection{Example 2}
In the second example, we solve a problem with discontinuous Lam\'{e}
parameters. We choose $\lambda=\mu=1$ for $0 < x < 0.5$ and
$\lambda=\mu=10$ for $0.5 < x < 1$. The solution
\begin{align*}
u = \begin{pmatrix} x^2y^3 - x^2y^3\sin(\pi x) \\
x^2y^3 - x^2y^3\sin(\pi x) \end{pmatrix}
\end{align*}
is chosen to be continuous with continuous normal stress and rotation
at $x=0.5$. Convergence rates are provided in Tables
\ref{tab:t3} and \ref{tab:t4}. The computed solution is plotted in Figure~\ref{fig:2}.
\begingroup
\def\arraystretch{1.1}
\begin{table}[H]
\scriptsize
\begin{center}
\begin{tabular}{c|cc|cc|cc|cc|cc|cc|cc}
     \hline
     & \multicolumn{2}{c|}{$\|\sigma - \sigma_h\| $ }  & \multicolumn{2}{c|}{$ \|\dvr(\sigma - \sigma_h)\| $} & \multicolumn{2}{c|}{$ \|u - u_h\|$} & \multicolumn{2}{c|}{$ \|\P_hu - u_h\| $}  &
     \multicolumn{2}{c|}{$\|\g - \g_h\| $} & \multicolumn{2}{c|}{$ \|u - \lambda_H\|_{a}$} & \multicolumn{2}{c}{CG iter.}
     \\
     $h$	&	error	&        rate   &	error	&	rate	&	error	&	rate	&	error	&	rate
     &	error	&	rate	&	error	&	rate	&	\#	&	rate             \\
     \hline
     1/4	&	2.02E-1	&	-	&	5.64E-1	&	-	&	4.57E-1	&	-	&	2.54E-1	&	-
     &	4.08E-1	&	-	&	5.01E-1	&	-	&	45	&	-		\\
     1/8	&	5.43E-2	&	1.9	&	2.98E-1	&	0.9	&	2.12E-1	&	1.1	&	7.14E-2	&	1.8
     &	1.04E-1	&	2.0	&	1.33E-1	&	1.9	&	61	&	-0.4		\\
     1/16	&	1.37E-2	&	2.0	&	1.51E-1	&	1.0	&	1.04E-1	&	1.0	&	1.84E-2	&	2.0
     &	2.60E-2	&	2.0	&	3.25E-2	&	2.0	&	85	&	-0.5		\\
     1/32	&	3.42E-3	&	2.0	&	7.58E-2	&	1.0	&	5.15E-2	&	1.0	&	4.63E-3	&	2.0
     &	6.47E-3	&	2.0	&	7.83E-3	&	2.1	&	122	&	-0.5		\\
     1/64	&	8.53E-4	&	2.0	&	3.79E-2	&	1.0	&	2.57E-2	&	1.0	&	1.16E-3	&	2.0
     &	1.61E-3	&	2.0	&	1.88E-3	&	2.1	&	170	&	-0.5		\\
     1/128	&	2.13E-4	&	2.0	&	1.90E-2	&	1.0	&	1.28E-2	&	1.0	&	2.90E-4	&	2.0
     &	4.02E-4	&	2.0	&	4.55E-4	&	2.1	&	252	&	-0.6		\\
     1/256	&	5.33E-5	&	2.0	&	9.48E-3	&	1.0	&	6.42E-3	&	1.0	&	7.25E-5	&	2.0
     &	1.00E-4	&	2.0	&	1.10E-4	&	2.0	&	354	&	-0.5		\\
     \hline
\end{tabular}
\end{center}
\caption{Numerical errors, convergence rates, and number of CG iterations with discontinuous quadratic mortars ($m=2$) for Example 2.}
\label{tab:t3}
\end{table}
\endgroup
\begingroup
\def\arraystretch{1.1}
\begin{table}[H]
\scriptsize
\begin{center}
\begin{tabular}{c|cc|cc|cc|cc|cc|cc|cc}
     \hline
     & \multicolumn{2}{c|}{$\|\sigma - \sigma_h\| $ }  & \multicolumn{2}{c|}{$ \|\dvr(\sigma - \sigma_h)\| $} & \multicolumn{2}{c|}{$ \|u - u_h\|$} & \multicolumn{2}{c|}{$ \|\P_hu - u_h\| $}  &
     \multicolumn{2}{c|}{$\|\g - \g_h\| $} & \multicolumn{2}{c|}{$ \|u - \lambda_H\|_{a}$} & \multicolumn{2}{c}{CG iter.}
     \\
     $h$	&	error	&        rate   &	error	&	rate	&	error	&	rate	&	error	&	rate
     &	error	&	rate	&	error	&	rate	&	\#	&	rate             \\
     \hline
     1/4	&	2.04E-1	&	-	&	5.64E-1	&	-	&	4.58E-1	&	-	&	2.54E-1	&	-
     &	4.04E-1	&	-	&	5.11E-1	&	-	&	52	&	-			\\
     1/16	&	1.37E-2	&	1.9	&	1.51E-1	&	1.0	&	1.04E-1	&	1.1	&	1.85E-2	&	1.9
     &	2.62E-2	&	2.0	&	3.27E-2	&	2.0	&	83	&	-0.3			\\
     1/64	&	8.68E-4	&	2.0	&	3.79E-2	&	1.0	&	2.57E-2	&	1.0	&	1.16E-3	&	2.0
     &	1.71E-3	&	2.0	&	1.90E-3	&	2.1	&	135	&	-0.4			\\
     1/256	&	5.51E-5	&	2.0	&	9.48E-3	&	1.0	&	6.42E-3	&	1.0	&	7.23E-5	&	2.0
     &	1.15E-4	&	2.0	&	1.19E-4	&	2.0	&	211	&	-0.3			\\	\hline
\end{tabular}
\end{center}
\caption{Numerical errors, convergence rates, and number of CG iterations with discontinuous cubic mortars ($m=3$) for Example 2.}
\label{tab:t4}
\end{table}
\endgroup
\begin{figure}[ht!]
\centering
\begin{subfigure}[b]{0.24\textwidth}
\includegraphics[width=\textwidth]{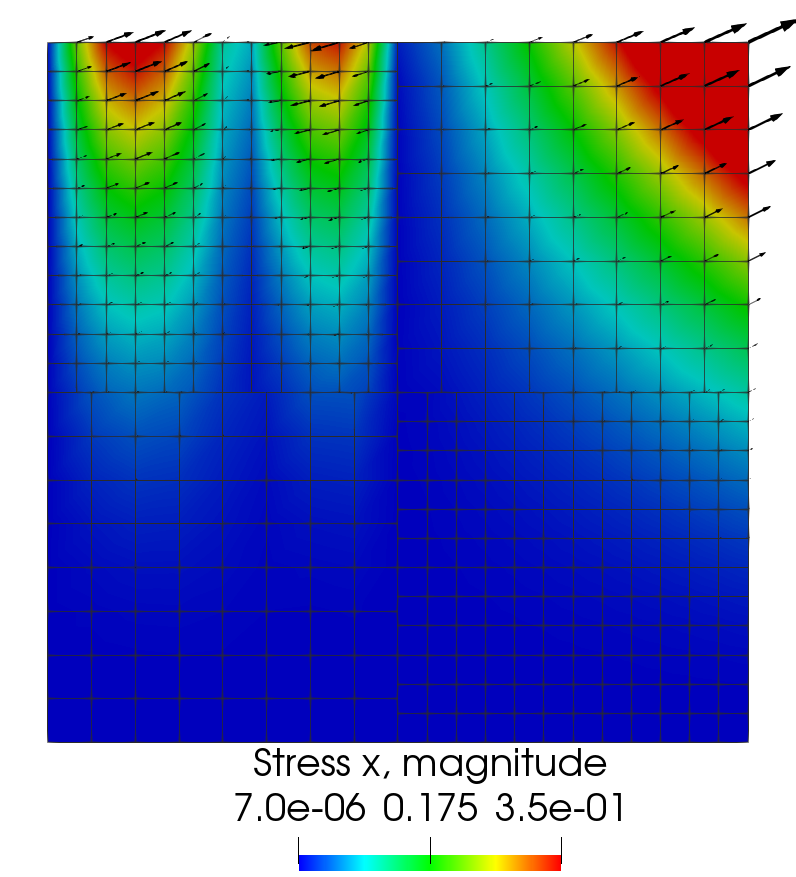}
\label{fig:2_1}
\end{subfigure}
\begin{subfigure}[b]{0.24\textwidth}
\includegraphics[width=\textwidth]{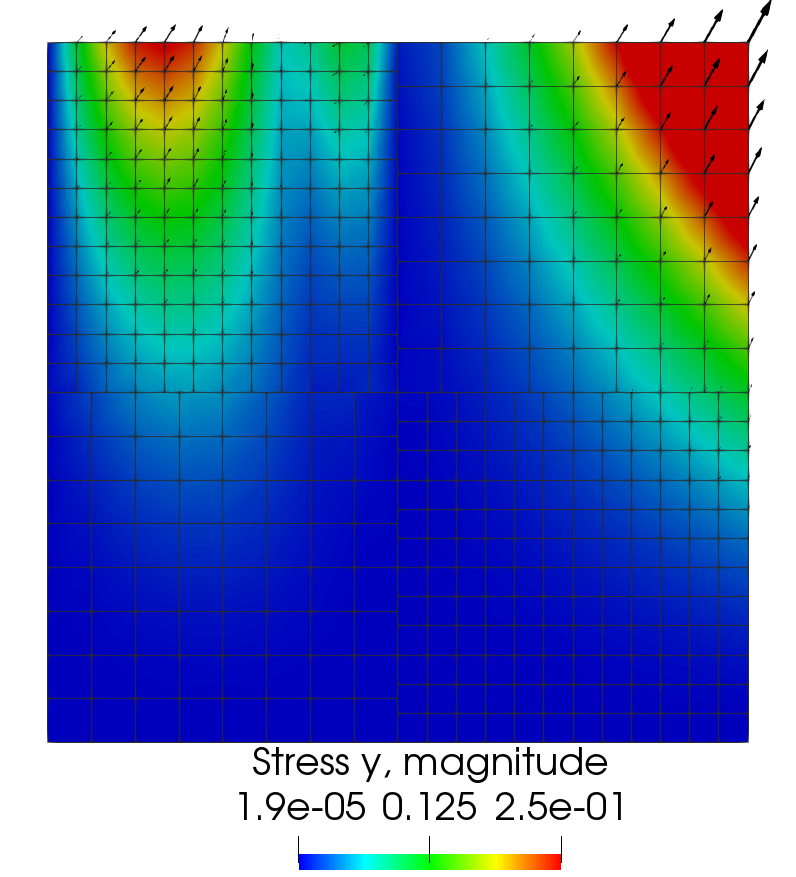}		
\label{fig:2_2}
\end{subfigure}
\begin{subfigure}[b]{0.24\textwidth}
\includegraphics[width=\textwidth]{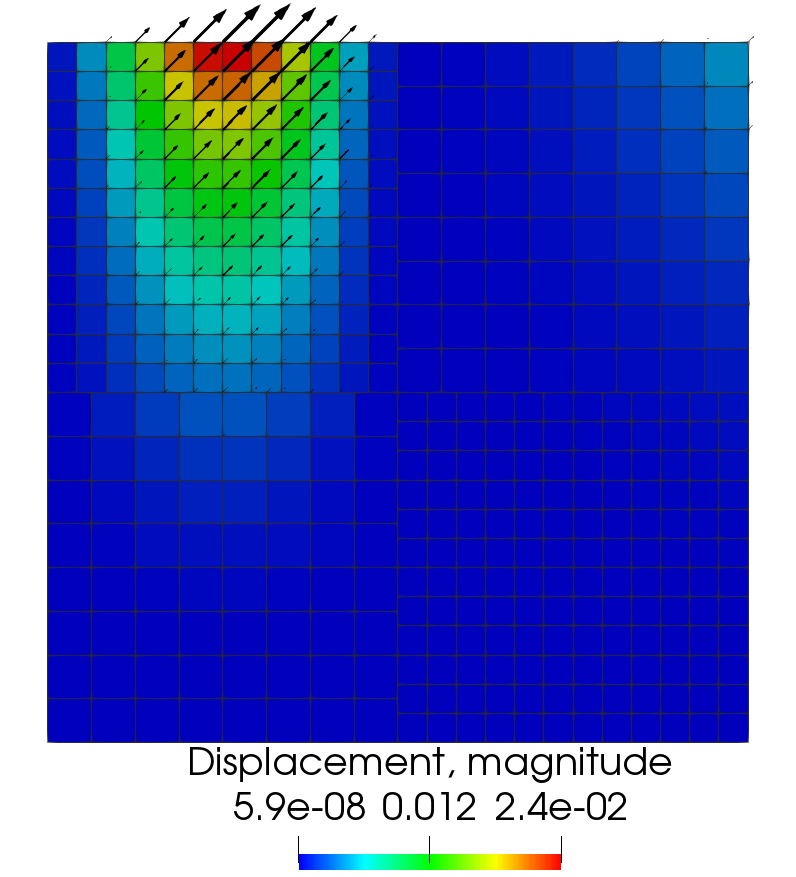}
\label{fig:2_3}
\end{subfigure}
\begin{subfigure}[b]{0.24\textwidth}
\includegraphics[width=\textwidth]{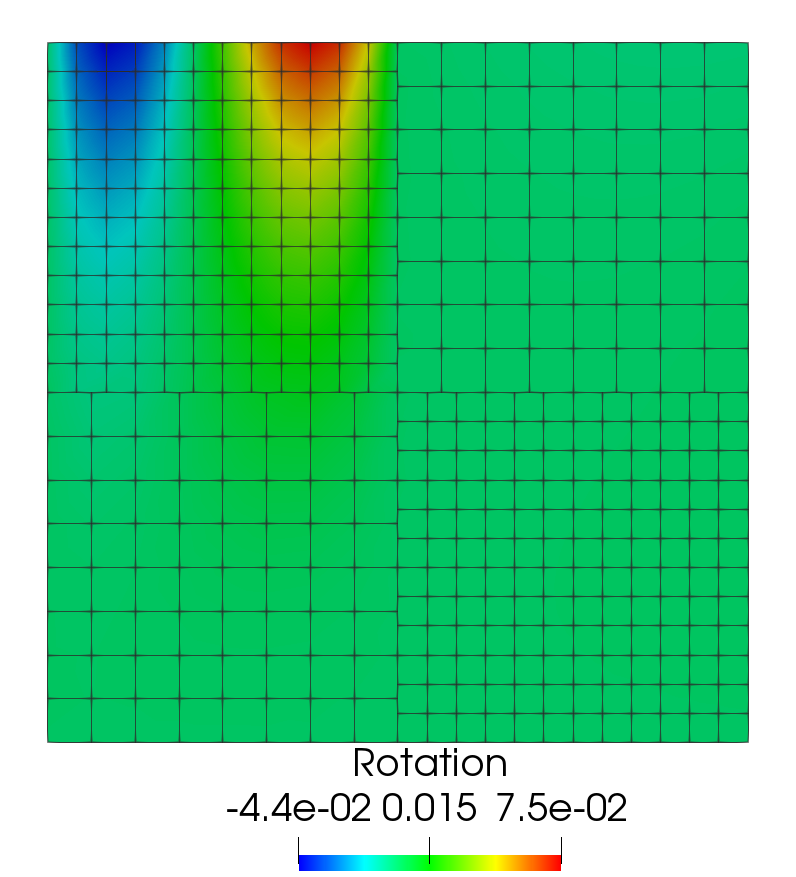}
\label{fig:2_4}
\end{subfigure}
\caption{Computed solution for Example 2, $h=1/16$.}\label{fig:2}
\end{figure}
\subsection{Example 3}
In third example we study a three-dimensional problem, which models
simultaneous twisting and compression (about $x$-axis) of the unit cube. The displacement solution is
$$
u = \begin{pmatrix} -0.1(e^x - 1)\sin(\pi x)\sin(\pi y) \\
-(e^x - 1)(y - \cos(\frac{\pi}{12})(y-0.5) + \sin(\frac{\pi}{12})(z-0.5)-0.5) \\
-(e^x - 1)(z - \sin(\frac{\pi}{12})(y-0.5) - \cos(\frac{\pi}{12})(z-0.5)-0.5)
\end{pmatrix}.
$$
The Lam\'{e} parameters are $\lambda = \mu = 100$.
The computed relative errors, convergence rates, and the number of interface iterations are shown in Table~\ref{tab:t5}.
We note that the mortar displacement exhibits slightly higher convergence rate than the theoretical rate.
The computed solution is plotted in Figure~\ref{fig:3}.
\begingroup
\def\arraystretch{1.1}
\begin{table}[H]
\scriptsize
\begin{center}
\begin{tabular}{c|cc|cc|cc|cc|cc|cc|cc}
     \hline
     & \multicolumn{2}{c|}{$\|\sigma - \sigma_h\| $ }  & \multicolumn{2}{c|}{$ \|\dvr(\sigma - \sigma_h)\| $} & \multicolumn{2}{c|}{$ \|u - u_h\|$} & \multicolumn{2}{c|}{$ \|\P_hu - u_h\| $}  &
     \multicolumn{2}{c|}{$\|\g - \g_h\| $} & \multicolumn{2}{c|}{$ \|u - \lambda_H\|_{a}$} & \multicolumn{2}{c}{CG iter.}
     \\
     $h$	&	error	&        rate   &	error	&	rate	&	error	&	rate	&	error	&	rate
     &	error	&	rate	&	error	&	rate	&	\#	&	rate             \\
     \hline
     1/4	&	2.71E-1	&	-	&	3.85E-1	&	-	&	2.60E-1	&	-	&	3.87E-2	&	-
     &	1.37E-1	&	-	&	2.80E-2	&	-	&	21	&	-		\\
     1/8	&	1.22E-1	&	1.2	&	1.96E-1	&	1.0	&	1.31E-1	&	1.0	&	8.40E-3	&	2.2
     &	6.83E-2	&	1.0	&	7.99E-3	&	1.8	&	37	&	-0.8		\\
     1/16	&	5.79E-2	&	1.1	&	9.87E-2	&	1.0	&	6.54E-2	&	1.0	&	2.09E-3	&	2.0
     &	3.41E-2	&	1.0	&	2.39E-3	&	1.7	&	56	&	-0.6		\\
     1/32	&	2.82E-2	&	1.0	&	4.94E-2	&	1.0	&	3.27E-2	&	1.0	&	5.31E-4	&	2.0
     &	1.71E-2	&	1.0	&	8.18E-4	&	1.6	&	80	&	-0.5		\\	\hline
\end{tabular}
\end{center}
\caption{Numerical errors, convergence rates, and number of CG iterations with discontinuous linear mortars ($m=1$) for Example 3.}
\label{tab:t5}
\end{table}
\endgroup
\begin{figure}[ht!]
\centering
\begin{subfigure}[b]{0.19\textwidth}
\includegraphics[width=\textwidth]{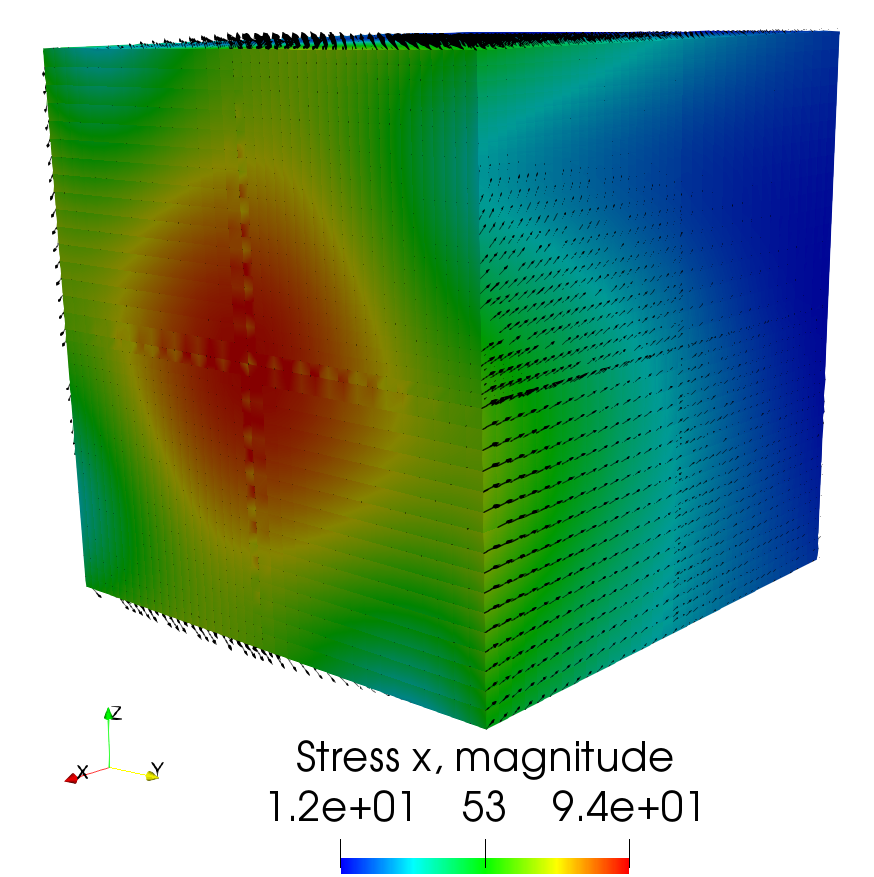}
\label{fig:3_1}
\end{subfigure}
\begin{subfigure}[b]{0.19\textwidth}
\includegraphics[width=\textwidth]{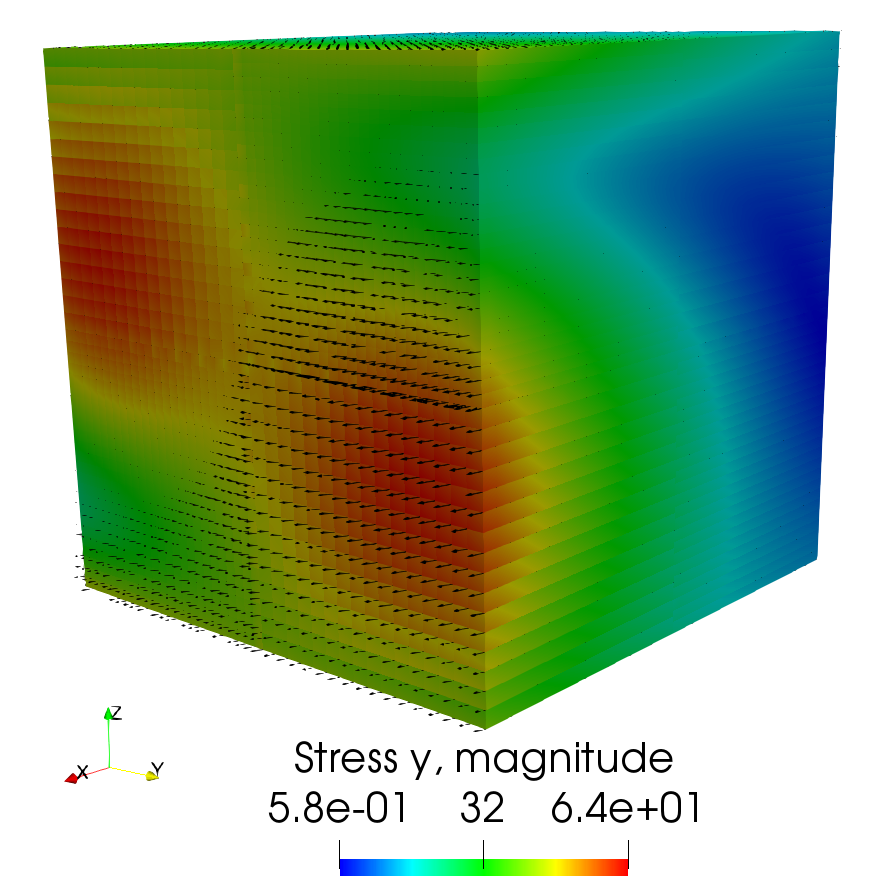}		
\label{fig:3_2}
\end{subfigure}
\begin{subfigure}[b]{0.19\textwidth}
\includegraphics[width=\textwidth]{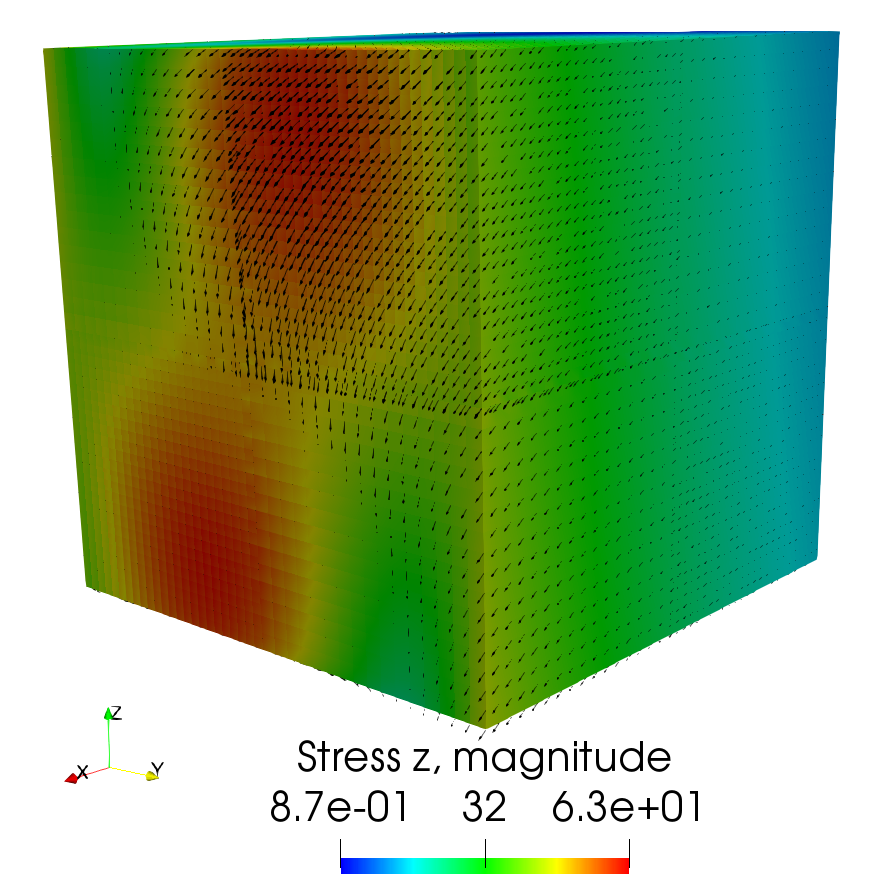}		
\label{fig:3_3}
\end{subfigure}
\begin{subfigure}[b]{0.19\textwidth}
\includegraphics[width=\textwidth]{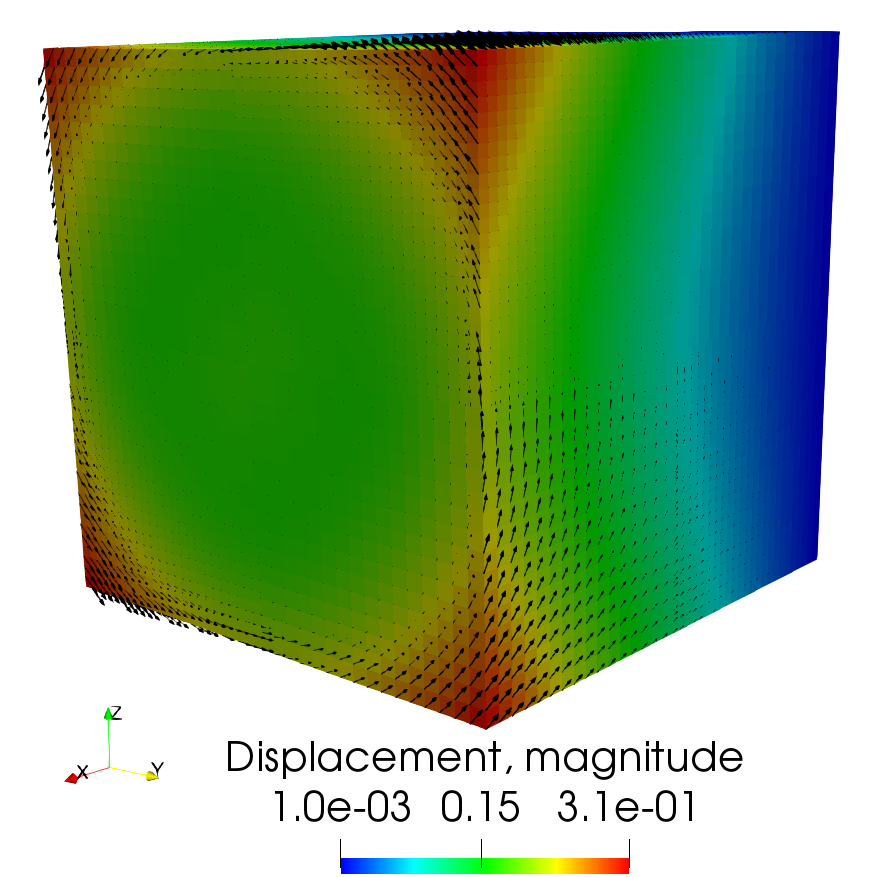}
\label{fig:3_4}
\end{subfigure}
\begin{subfigure}[b]{0.19\textwidth}
\includegraphics[width=\textwidth]{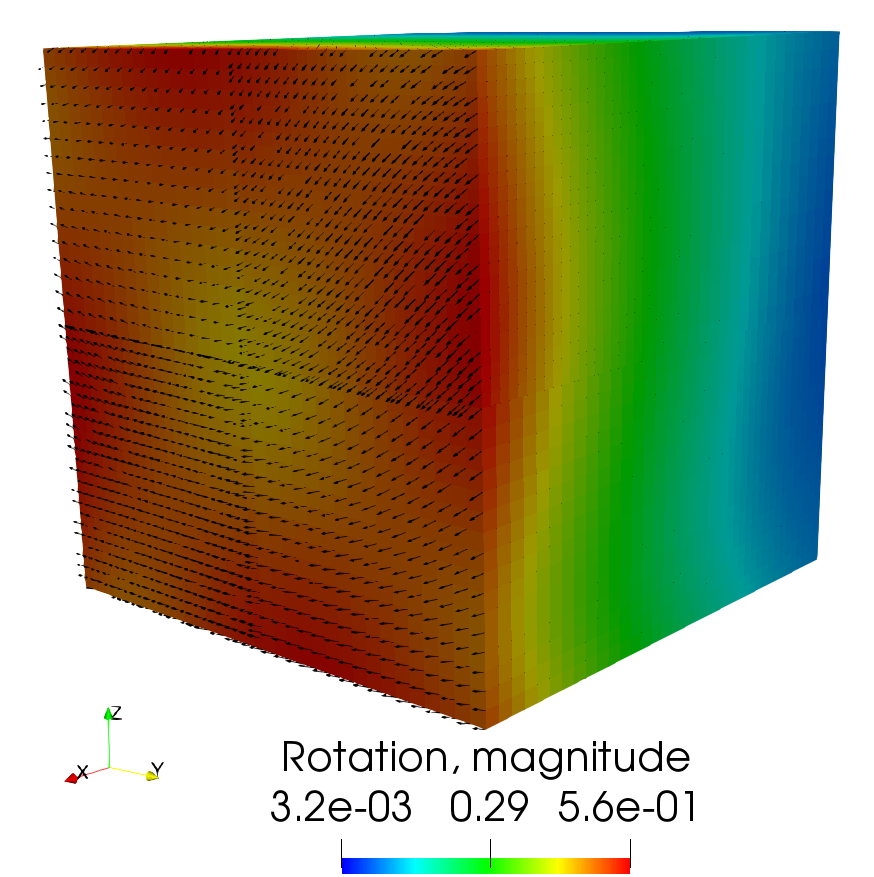}
\label{fig:3_5}
\end{subfigure}
\caption{Computed solution for Example 3, $h=1/32$.}\label{fig:3}
\end{figure}
\subsection{Example 4}
In this example we study the dependence of the
number of CG iterations on the number of subdomains used for solving the
problem. We consider the same test case as in Example 1 with
discontinuous quadratic mortars, but solve the
problem using $2\times 2$, $4\times 4$ and $8\times 8$ subdomain
partitionings. We report the number of CG iterations in Table
\ref{tab:t6}. For the sake of space and clarity we do not show the
rate of growth for each refinement step, but only the
average values. For each fixed domain decomposition (each column)
we observe growth of $\mathcal{O}(h^{-0.5})$ as the grids are refined,
confirming condition number $\kappa = \mathcal{O}(h^{-1})$, as in the previous
examples with $2\times 2$ decompositions. Considering each row,
we observe that the number of CG iterations grows as the subdomain
size $A$ decreases with rate $\mathcal{O}(A^{-0.5})$, implying that
$\kappa = \mathcal{O}(A^{-1})$. This is expected for an algorithm
without a coarse solve preconditioner \cite{Toselli-Widlund}. This
issue will be addressed in forthcoming work.
\begingroup
\def\arraystretch{1.2}
\begin{table}[H]
\scriptsize
\begin{center}
\begin{tabular}{c|c|c|c|c}
     \hline
     $h$
     & $2 \times 2$
     & $4 \times 4$
     & $8 \times 8$
     & Rate \\
     \hline
     1/16	&	48		&	67	&	94	& $\mathcal{O}(A^{-0.5})$	\\
     1/32	&	63		&	94	&	118	& $\mathcal{O}(A^{-0.5})$	\\
     1/64	&	96		&	133	&	167	& $\mathcal{O}(A^{-0.4})$	\\
     1/128	&	136		&	189	&	230	& $\mathcal{O}(A^{-0.4})$	\\
     1/256	&	194		&	267	&	340     & $\mathcal{O}(A^{-0.4})$	\\ \hline
     Rate	&	$\mathcal{O}(h^{-0.5})$	& $\mathcal{O}(h^{-0.5})$	& $\mathcal{O}(h^{-0.5})$	&\\ \hline
\end{tabular}
\end{center}
\caption{Number of CG iterations for Example 4.} \label{tab:t6}
\end{table}
\endgroup

\subsection{Example 5}

In the last example we test the efficiency of the multiscale stress
basis (MSB) technique outlined in the previous section. With no MSB
the total number of solves is $\#\text{CG iter.} + 3$, one for each CG
iteration plus one solve for the right hand side of type
\eqref{bar1-1}--\eqref{bar1-3}, one for the initial residual and one
to recover the final solution.  On the other hand, the method with MSB
requires $\text{dim}(\Lambda_H) + 3$ solves, hence its use is
advantageous when $\text{dim}(\Lambda_h)<\#\text{CG iter.}$, that is
when the mortar grid is relatively coarse.

We use a heterogeneous porosity field from the Society of Petroleum
Engineers (SPE) Comparative Solution
Project2\footnote{http://www.spe.org/csp}.  The computation domain is
$\O = (0,1)^2$ with a fixed rectangular $128\times 128$ grid. The left
and right boundary conditions are $u = (0.1,0)^T$ and $u =
(0,0)^T$. Zero normal stress, $\s\, n = 0$, is specified on the top
and bottom boundaries. Given the porosity $\phi$, the Young's modulus
is obtained from the relation \cite{kovavcik1999correlation} $
E = 10^2 \left( 1-\frac{\phi}{c} \right)^{2.1}$, where the constant $c
= 0.5$ refers to the porosity at which the effective Young's modulus
becomes zero. The choice of this constant is based on the properties
of the deformable medium, see \cite{kovavcik1999correlation} for
details. The resulting Young's modulus field is shown in Figure
\ref{fig:4}.

A comparison between the fine scale solution and the multiscale
solution with $8\times 8$ subdomains and a single cubic mortar per
interface is shown in Figure \ref{fig:4}. We observe that the two
solutions are very similar and that the multiscale solution captures
the heterogeneity very well, even for this very coarse mortar space.
In Table~\ref{tab:t7} we compare the cost of using MSB and not using
MSB for several choices of mortar grids. We report the number of
solves per subdomain, which is the dominant computational cost. We
conclude that for cases with relatively coarse mortar grids, the MSB
technique requires significantly fewer subdomain solves, resulting in
faster computations. Moreover, as evident from the last row in
Table~\ref{tab:t7}, computing the fine scale solution is significantly
more expensive than computing the multiscale solution.
\begin{table}[H]
\scriptsize
\begin{center}
\begin{tabular}{l|c|c|c}
     \hline
     Mortar type
     & $H$
     & \# Solves, no MSB
     & \# Solves, MSB \\
     \hline
     Quadratic					&	1/8			&	180	&	27	\\
     Cubic						&	1/8			&	173	&	35  \\
     Quadratic					&	1/16		&	219	&	51  \\
     Cubic						&	1/16		&	250	&	67  \\
     Linear (fine scale solution)&	1/128		&	295	&	195 \\ \hline
\end{tabular}
\end{center}
\caption{Number of subdomain solves for Example 5.} \label{tab:t7}
\end{table}
\begin{figure}[ht!]
\centering
\begin{subfigure}[b]{0.19\textwidth}
\includegraphics[width=\textwidth]{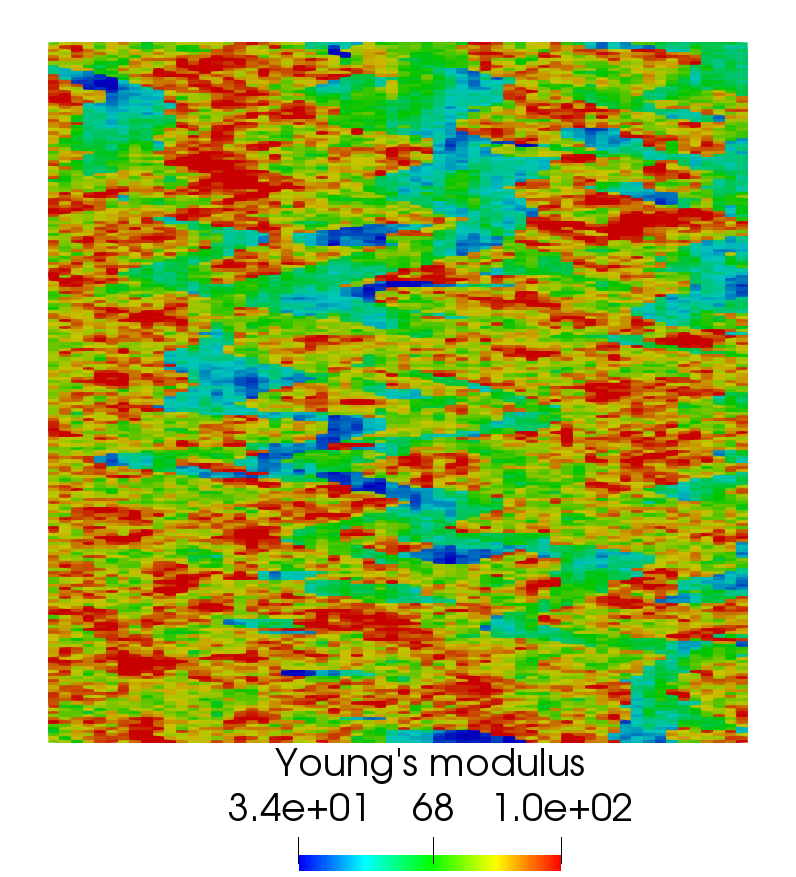}
\label{fig:4_1}
\end{subfigure}
\begin{subfigure}[b]{0.19\textwidth}
\includegraphics[width=\textwidth]{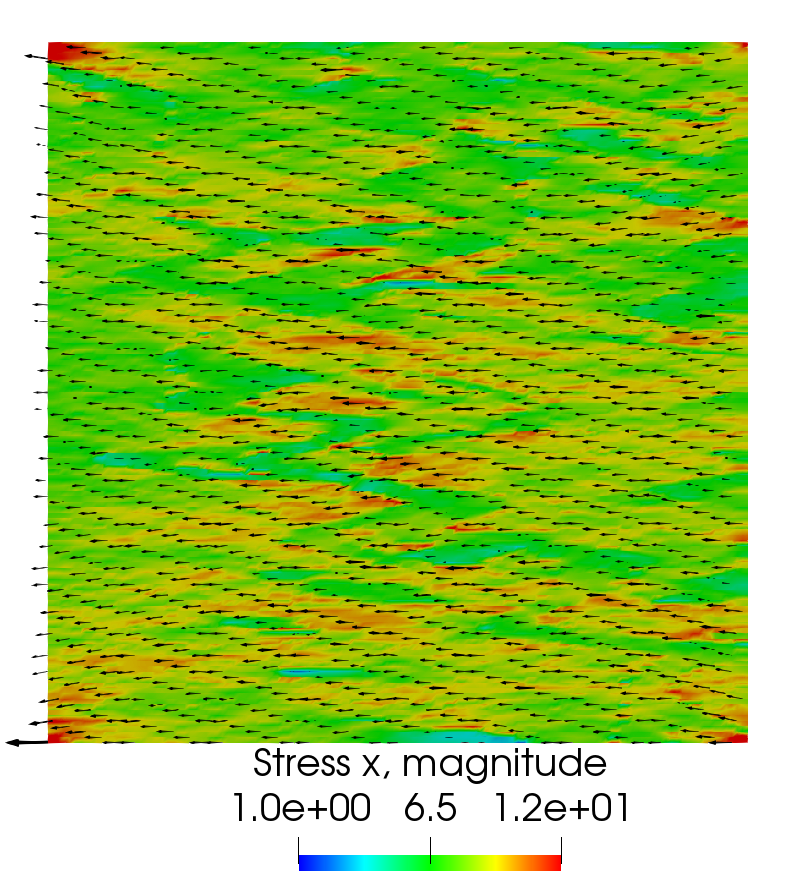}		
\label{fig:4_2}
\end{subfigure}
\begin{subfigure}[b]{0.19\textwidth}
\includegraphics[width=\textwidth]{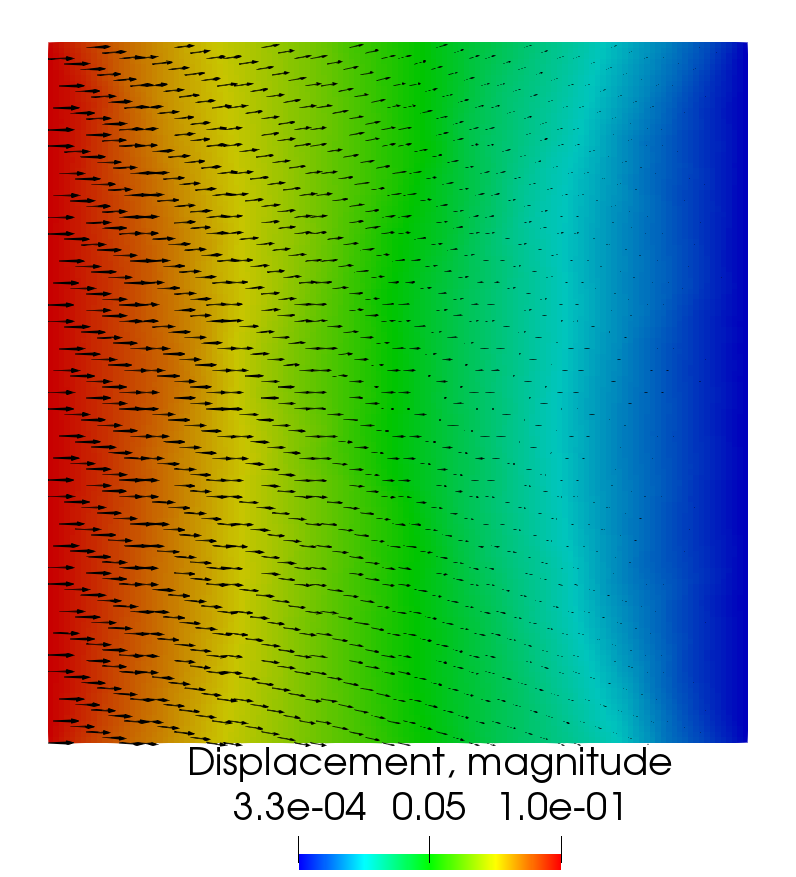}
\label{fig:4_4}
\end{subfigure}
\begin{subfigure}[b]{0.19\textwidth}
\includegraphics[width=\textwidth]{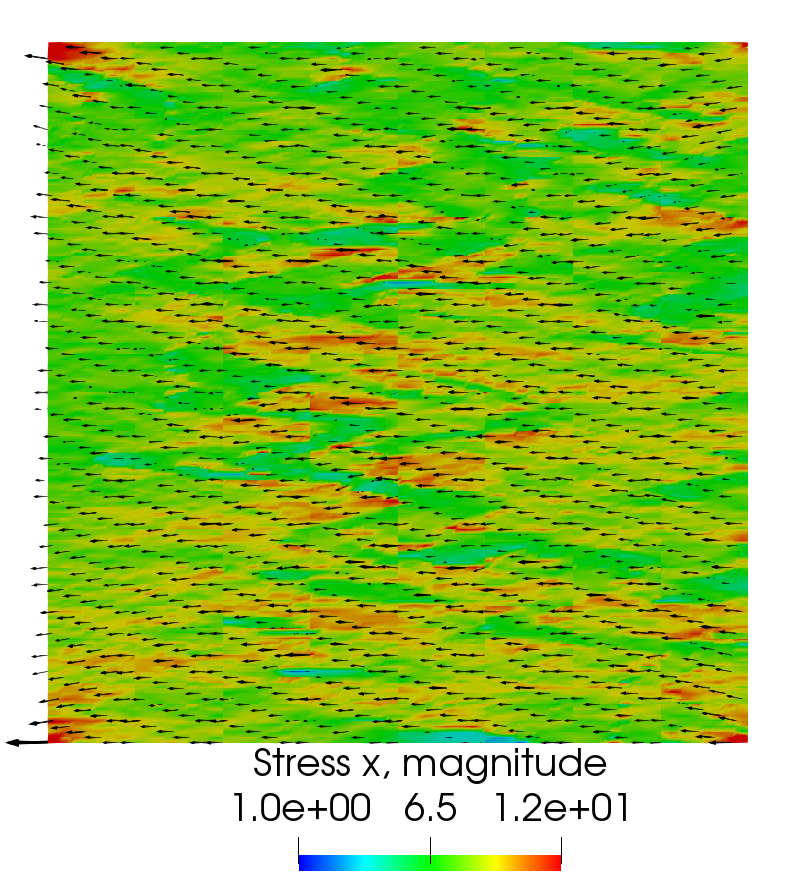}		
\label{fig:4_22}
\end{subfigure}
\begin{subfigure}[b]{0.19\textwidth}
\includegraphics[width=\textwidth]{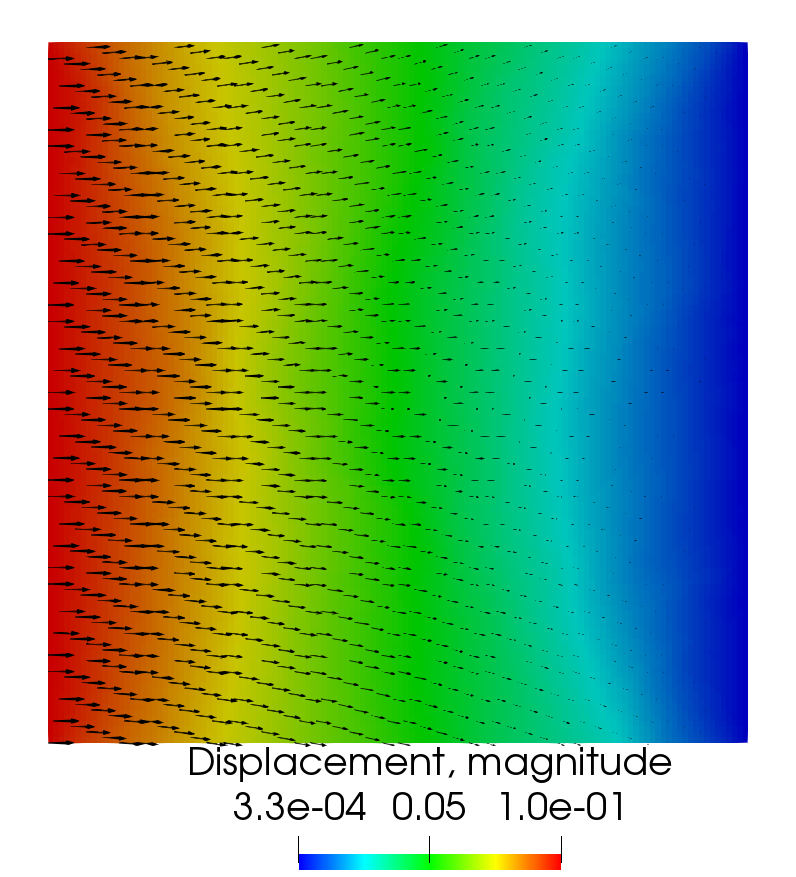}
\label{fig:4_42}
\end{subfigure}
\caption{Example 5, Young's modulus, fine scale stress and displacement,
and multiscale stress and displacement with cubic mortars, $H = 1/8$.}\label{fig:4}
\end{figure}

\bibliographystyle{siamplain}
\bibliography{elastdd}

\begin{thebibliography}{10}

\bibitem{Amara-Thomas}
{\sc M.~Amara and J.~M. Thomas}, {\em Equilibrium finite elements for the
  linear elastic problem}, Numer. Math., 33 (1979), pp.~367--383.

\bibitem{MSMFEM-1}
{\sc I.~Ambartsumyan, E.~Khattatov, J.~Nordbotten, and I.~Yotov}, {\em A
  multipoint stress mixed finite element method for elasticity {I}:
  {S}implicial grids}.
\newblock Preprint.

\bibitem{MSMFEM-2}
{\sc I.~Ambartsumyan, E.~Khattatov, J.~Nordbotten, and I.~Yotov}, {\em A
  multipoint stress mixed finite element method for elasticity {II}:
  {Q}uadrilateral grids}.
\newblock Preprint.

\bibitem{arbogast2000mixed}
{\sc T.~Arbogast, L.~C. Cowsar, M.~F. Wheeler, and I.~Yotov}, {\em Mixed finite
  element methods on nonmatching multiblock grids}, SIAM J. Numer. Anal., 37
  (2000), pp.~1295--1315.

\bibitem{APWY}
{\sc T.~Arbogast, G.~Pencheva, M.~F. Wheeler, and I.~Yotov}, {\em A multiscale
  mortar mixed finite element method}, Multiscale Model. Simul., 6 (2007),
  pp.~319--346.

\bibitem{dealii}
{\sc D.~Arndt, W.~Bangerth, D.~Davydov, T.~Heister, L.~Heltai, M.~Kronbichler,
  M.~Maier, J.-P. Pelteret, B.~Turcksin, and D.~Wells}, {\em The
  \texttt{deal.II} library, version 8.5}, J. Numer. Math., 25 (2017),
  pp.~137--146.

\bibitem{ArnAwaQiu}
{\sc D.~N. Arnold, G.~Awanou, and W.~Qiu}, {\em Mixed finite elements for
  elasticity on quadrilateral meshes}, Adv. Comput. Math., 41 (2015),
  pp.~553--572.

\bibitem{PEERS}
{\sc D.~N. Arnold, F.~Brezzi, and J.~Douglas, Jr.}, {\em P{EERS}: a new mixed
  finite element for plane elasticity}, Japan J. Appl. Math., 1 (1984),
  pp.~347--367.

\bibitem{arnold2007mixed}
{\sc D.~N. Arnold, R.~S. Falk, and R.~Winther}, {\em Mixed finite element
  methods for linear elasticity with weakly imposed symmetry}, Math. Comp., 76
  (2007), pp.~1699--1723.

\bibitem{arnold2014mixed}
{\sc D.~N. Arnold and J.~J. Lee}, {\em Mixed methods for elastodynamics with
  weak symmetry}, SIAM J. Numer. Anal., 52 (2014), pp.~2743--2769.

\bibitem{Awanou-rect-weak}
{\sc G.~Awanou}, {\em Rectangular mixed elements for elasticity with weakly
  imposed symmetry condition}, Adv. Comput. Math., 38 (2013), pp.~351--367.

\bibitem{boffi2009reduced}
{\sc D.~Boffi, F.~Brezzi, and M.~Fortin}, {\em Reduced symmetry elements in
  linear elasticity}, Commun. Pure Appl. Anal., 8 (2009), pp.~95--121.

\bibitem{brezzi1991mixed}
{\sc F.~Brezzi and M.~Fortin}, {\em Mixed and hybrid finite element methods},
  vol.~15 of Springer Series in Computational Mathematics, Springer-Verlag, New
  York, 1991.

\bibitem{ciarlet2002finite}
{\sc P.~G. Ciarlet}, {\em The finite element method for elliptic problems},
  vol.~40 of Classics in Applied Mathematics, Society for Industrial and
  Applied Mathematics, Philadelphia, PA, 2002.

\bibitem{cockburn2010new}
{\sc B.~Cockburn, J.~Gopalakrishnan, and J.~Guzm\'an}, {\em A new elasticity
  element made for enforcing weak stress symmetry}, Math. Comp., 79 (2010),
  pp.~1331--1349.

\bibitem{cowsar1995balancing}
{\sc L.~C. Cowsar, J.~Mandel, and M.~F. Wheeler}, {\em Balancing domain
  decomposition for mixed finite elements}, Math. Comp., 64 (1995),
  pp.~989--1015.

\bibitem{Dauge}
{\sc M.~Dauge}, {\em Elliptic boundary value problems on corner domains},
  vol.~1341 of Lecture Notes in Mathematics, Springer-Verlag, Berlin, 1988.
\newblock Smoothness and asymptotics of solutions.

\bibitem{Efe-Gal-Hou}
{\sc Y.~Efendiev, J.~Galvis, and T.~Y. Hou}, {\em Generalized multiscale finite
  element methods ({GM}s{FEM})}, J. Comput. Phys., 251 (2013), pp.~116--135.

\bibitem{farhat1991method}
{\sc C.~Farhat and F.-X. Roux}, {\em A method of finite element tearing and
  interconnecting and its parallel solution algorithm}, Internat. J. Numer.
  Methods Engrg., 32 (1991), pp.~1205--1227.

\bibitem{Fritz-mortar-elast}
{\sc A.~Fritz, S.~H\"ueber, and B.~I. Wohlmuth}, {\em A comparison of mortar
  and {N}itsche techniques for linear elasticity}, Calcolo, 41 (2004),
  pp.~115--137.

\bibitem{galvis2007non}
{\sc J.~Galvis and M.~Sarkis}, {\em Non-matching mortar discretization analysis
  for the coupling {S}tokes-{D}arcy equations}, Electron. Trans. Numer. Anal.,
  26 (2007), pp.~350--384.

\bibitem{ganis2009implementation}
{\sc B.~Ganis and I.~Yotov}, {\em Implementation of a mortar mixed finite
  element method using a multiscale flux basis}, Comput. Methods Appl. Mech.
  Engrg., 198 (2009), pp.~3989--3998.

\bibitem{GPWW-2009}
{\sc V.~Girault, G.~V. Pencheva, M.~F. Wheeler, and T.~M. Wildey}, {\em Domain
  decomposition for linear elasticity with {DG} jumps and mortars}, Comput.
  Methods Appl. Mech. Engrg., 198 (2009), pp.~1751--1765.

\bibitem{GW}
{\sc R.~Glowinski and M.~F. Wheeler}, {\em Domain decomposition and mixed
  finite element methods for elliptic problems}, in First International
  Symposium on Domain Decomposition Methods for Partial Differential Equations,
  R.~G. et~al., ed., SIAM, Philadelphia, 1988, pp.~144--172.

\bibitem{Goldfeld-dd-elast-mixed}
{\sc P.~Goldfeld, L.~F. Pavarino, and O.~B. Widlund}, {\em Balancing
  {N}eumann-{N}eumann preconditioners for mixed approximations of heterogeneous
  problems in linear elasticity}, Numer. Math., 95 (2003), pp.~283--324.

\bibitem{gopalakrishnan2012second}
{\sc J.~Gopalakrishnan and J.~Guzm\'an}, {\em A second elasticity element using
  the matrix bubble}, IMA J. Numer. Anal., 32 (2012), pp.~352--372.

\bibitem{grisvard2011elliptic}
{\sc P.~Grisvard}, {\em Elliptic problems in nonsmooth domains}, vol.~69 of
  Classics in Applied Mathematics, Society for Industrial and Applied
  Mathematics (SIAM), Philadelphia, PA, 2011.

\bibitem{Hauret-LeTallec}
{\sc P.~Hauret and P.~Le~Tallec}, {\em A discontinuous stabilized mortar method
  for general 3{D} elastic problems}, Comput. Methods Appl. Mech. Engrg., 196
  (2007), pp.~4881--4900.

\bibitem{kelley1995iterative}
{\sc C.~T. Kelley}, {\em Iterative methods for linear and nonlinear equations},
  vol.~16 of Frontiers in Applied Mathematics, Society for Industrial and
  Applied Mathematics, Philadelphia, 1995.

\bibitem{Kim-elast-BDDC}
{\sc H.~H. Kim}, {\em A {BDDC} algorithm for mortar discretization of
  elasticity problems}, SIAM J. Numer. Anal., 46 (2008), pp.~2090--2111.

\bibitem{Kim-elast-FETI}
{\sc H.~H. Kim}, {\em A {FETI}-{DP} formulation of three dimensional elasticity
  problems with mortar discretization}, SIAM J. Numer. Anal., 46 (2008),
  pp.~2346--2370.

\bibitem{Klawonn-dd-elast}
{\sc A.~Klawonn and O.~B. Widlund}, {\em A domain decomposition method with
  {L}agrange multipliers for linear elasticity}, in Eleventh {I}nternational
  {C}onference on {D}omain {D}ecomposition {M}ethods ({L}ondon, 1998),
  Augsburg, 1999, pp.~49--56.

\bibitem{kovavcik1999correlation}
{\sc J.~Kovacik}, {\em Correlation between {Y}oung's modulus and porosity in
  porous materials}, J. Mater. Sci. Lett., 18 (1999), pp.~1007--1010.

\bibitem{lions2011non}
{\sc J.-L. Lions and E.~Magenes}, {\em Non-homogeneous boundary value problems
  and applications. {V}ol. {I}}, Springer-Verlag, New York-Heidelberg, 1972.

\bibitem{mathew1989domain}
{\sc T.~P. Mathew}, {\em Domain decomposition and iterative refinement methods
  for mixed finite element discretizations of elliptic problems}, PhD thesis,
  Courant Institute of Mathematical Sciences, New York University, 1989.
\newblock Tech. Rep. 463.

\bibitem{Pavarino-dd-elast-mixed}
{\sc L.~F. Pavarino, O.~B. Widlund, and S.~Zampini}, {\em B{DDC}
  preconditioners for spectral element discretizations of almost incompressible
  elasticity in three dimensions}, SIAM J. Sci. Comput., 32 (2010),
  pp.~3604--3626.

\bibitem{pencheva2003balancing}
{\sc G.~Pencheva and I.~Yotov}, {\em Balancing domain decomposition for mortar
  mixed finite element methods}, Numer. Linear Algebra Appl., 10 (2003),
  pp.~159--180.

\bibitem{PWY}
{\sc M.~Peszy{\'n}ska, M.~F. Wheeler, and I.~Yotov}, {\em Mortar upscaling for
  multiphase flow in porous media}, Comput. Geosci., 6 (2002), pp.~73--100.

\bibitem{QV}
{\sc A.~Quarteroni and A.~Valli}, {\em Domain Decomposition Methods for Partial
  Differential equations}, Clarendon Press, Oxford, 1999.

\bibitem{roberts1991mixed}
{\sc J.~E. Roberts and J.-M. Thomas}, {\em Mixed and hybrid methods}, in
  Handbook of numerical analysis, {V}ol.\ {II}, Handb. Numer. Anal., II,
  North-Holland, Amsterdam, 1991, pp.~523--639.

\bibitem{Scott-Zhang}
{\sc L.~R. Scott and S.~Zhang}, {\em Finite element interpolation of nonsmooth
  functions satisfying boundary conditions}, Math. Comput., 54 (1990),
  pp.~483--493.

\bibitem{stenberg1988family}
{\sc R.~Stenberg}, {\em A family of mixed finite elements for the elasticity
  problem}, Numer. Math., 53 (1988), pp.~513--538.

\bibitem{Toselli-Widlund}
{\sc A.~Toselli and O.~Widlund}, {\em Domain decomposition methods---algorithms
  and theory}, vol.~34 of Springer Series in Computational Mathematics,
  Springer-Verlag, Berlin, 2005.

\bibitem{VasWangYot}
{\sc D.~Vassilev, C.~Wang, and I.~Yotov}, {\em Domain decomposition for coupled
  {S}tokes and {D}arcy flows}, Comput. Methods Appl. Mech. Engrg., 268 (2014),
  pp.~264--283.

\bibitem{wheeler2012multiscale}
{\sc M.~F. Wheeler, G.~Xue, and I.~Yotov}, {\em A multiscale mortar multipoint
  flux mixed finite element method}, ESAIM Math. Model. Numer. Anal., 46
  (2012), pp.~759--796.

\end{thebibliography}
\end{document}